\documentclass{TPmod}
\usepackage{url}
\usepackage{setspace}
\usepackage{scrextend}
\usepackage{mathtools}

\usepackage{tikz}
\usepackage[english]{babel}
\usepackage[utf8]{inputenc}
\usetikzlibrary{arrows,decorations.pathmorphing,backgrounds,fit,positioning,shapes.symbols,chains}

\usetikzlibrary{calc}
\usepackage{amssymb}

\usepackage{amsthm}
\usepackage{amssymb}
\usepackage[capitalize]{cleveref}

\newtheorem{ass}[thm]{Assumption}

\newcommand{\Coh}{\mathrm{Coh}}

\DeclareMathOperator{\Aff}{Aff}
\DeclareMathOperator{\Cob}{Cob}
\DeclareMathOperator{\CH}{CH}

\def\bP{\mathbb{P}}

\def\bZ{\mathbb{Z}}
\def\bL{\mathbb{L}}

\def\ints{\mathsf{s}}

\def\AJ{\mathrm{AJ}}
\def\Alb{\mathrm{Alb}}
\def\alb{\mathrm{alb}}
\def\Jac{\mathrm{Jac}}
\def\hom{\mathrm{hom}}
\def\tri{\mathrm{tri}}
\def\GL{\mathrm{GL}}
\def\Trop{\mathrm{Trop}}
\def\trop{\mathrm{trop}}

\def\scrA{\EuScript{A}}

\def\scrY{\EuScript{Y}}
\def\scrM{\EuScript{M}}
\def\scrW{\EuScript{W}}
\def\unob{\mathrm{unob}}
\def\SL{\mathrm{SL}}
\def\loc{\mathrm{loc}}
\def\fib{\mathrm{fib}}

\newcommand{\scrG}{\EuScript{G}}
\newcommand{\scrE}{\EuScript{E}}
\newcommand{\scrF}{\EuScript{F}}

\begin{document}

\title{Lagrangian cobordism and tropical curves}

\author[Sheridan and Smith]{Nick Sheridan and Ivan Smith}

\address{Nick Sheridan, School of Mathematics, University of Edinburgh, Edinburgh EH9 3FD, U.K.}
\address{Ivan Smith, Centre for Mathematical Sciences, University of Cambridge, Wilberforce Road, Cambridge CB3 0WB, U.K.}

\begin{abstract} {\sc Abstract:} We study a cylindrical Lagrangian cobordism group for Lagrangian torus fibres in symplectic manifolds which are the total spaces of smooth Lagrangian torus fibrations.  We use ideas from family Floer theory and tropical geometry to obtain both obstructions to and constructions of cobordisms; in particular, we give examples of symplectic tori in which the cobordism group has no non-trivial cobordism relations between pairwise distinct fibres, and ones in which the degree zero fibre  cobordism group is a divisible group. The results are independent of but motivated by mirror symmetry, and a relation to rational equivalence of $0$-cycles on the mirror rigid analytic space.
\end{abstract}
\maketitle


\section{Introduction} 

\subsection{(Non)-existence theorems}

Classifying Lagrangian submanifolds of a symplectic manifold $(X,\omega)$ up to Hamiltonian isotopy is a notoriously difficult problem, resolved in only a handful of cases. A coarser equivalence relation (originally introduced by Arnol'd and systematically studied by Biran and Cornea) is to ask for the classification up to Lagrangian cobordism. Lagrangian submanifolds $L^-$ and $L^+$  of a symplectic manifold $(X,\omega)$ are Lagrangian cobordant if there is a Lagrangian submanifold $V\subset (\C\times X, \omega_{\C} \oplus \omega)$ which agrees with $\R_- \times L^- \cup \R_+ \times L^+$ outside a compact subset. The definition extends straightforwardly to a cobordism relation between tuples of Lagrangians $\bL^-$ and $\bL^+$, where the cobordism should fibre over parallel half-lines in a neighbourhood of negative and positive real infinity. 

Without any further constraints, the $h$-principle implies that Lagrangian cobordism is an essentially topological relation: more precisely, Lagrangian immersions are governed by an $h$-principle, and the double points of a Lagrangian immersion may be eliminated (changing the global topology, and often losing orientability) by Lagrange surgery.  If, however,  one asks that the cobordisms be geometrically constrained --  exact, monotone, or more generally Floer-theoretically unobstructed -- then striking rigidity phenomena appear. 

The \emph{planar Lagrangian cobordism group} $\Cob(X/\C)$ is usually defined as the abelian group $\Cob(X/\C)$ of formal sums of Lagrangian submanifolds, modulo relations arising from cobordisms. 
However it is natural to consider Lagrangians decorated with certain additional structures, which are called Lagrangian branes. 
Specifically, in the language of \cite{Seidel:graded}, we  assume $c_1(X)=0$, fix an $\infty$-fold Maslov cover $\mathcal{L}$ of the oriented Lagrangian Grassmannian of $X$, and define a Lagrangian brane to be a Lagrangian equipped with a lift to $\mathcal{L}$. 
So we define $\Cob(X/\C)$ to be the free abelian group of formal sums of Lagrangian branes, modulo relations arising from cobordisms equipped with brane structures, cf. Section \ref{Sec:lag_cobordism}. 
Of course the group depends on $\omega$ and $\mathcal{L}$, but we suppress them from the notation. 
Note that our branes come with orientations, so there is a well-defined map $\Cob(X/\C) \to H_n(X;\Z)$, whose kernel we denote by $\Cob(X/\C)_{\hom}$.

\begin{rmk}
There are many possible variations on the data that constitute a brane structure: (relative) spin structures if one wants to define Floer theory in characteristic not equal to $2$, local systems, bounding cochains. 
We will be clear which notion of `brane' we are working with at each stage.
\end{rmk}

For reasons which will be explained in Section \ref{Sec:Cob_Rat_Eq}, in this paper we in fact focus on a variation called the \emph{cylindrical} cobordism group, $\Cob(X/\C^*)$. 
This is almost the same as the planar cobordism group, except the relations come from cobordisms in $\C^* \times X$ rather than $\C \times X$. 

We investigate Lagrangian cobordism relations amongst Lagrangian torus fibres of a symplectic manifold which is the total space of an integrable system with nonsingular fibres.  More concretely, let $B$ be an \emph{oriented tropical affine manifold}, i.e. a manifold with an atlas of charts whose transition functions belong to $\R^n \rtimes \SL(n,\Z)$. Let $X(B) \coloneqq T^*B/T^*_{\Z}B$, and equip $X(B)$ with its canonical symplectic form induced from the exact form on $T^*B$.  Then $X(B)$ is the total space of a Lagrangian torus fibration over $B$; let $F_b$ denote the Lagrangian torus fibre over $b\in B$. 

Because $B$ is oriented, $X(B)$ comes with a natural homotopy class of holomorphic volume forms and hence a natural Maslov cover $\mathcal{L}$, and the fibres $F_b$ come with natural brane structures. 
The \emph{cylindrical Lagrangian cobordism group of fibres} in $X(B)$ is the additive group $\Cob_{\fib}(X(B)/\C^*)$ generated by the branes $[F_b]$, modulo relations arising from cylindrical Lagrangian brane cobordisms which carry fibres over the ends, cf. Section \ref{Sec:lag_cobordism}. 

We start with an existence theorem for such cobordisms. 
Let $B = B(M)  = \R^n / M\cdot \Z^n$ be a tropical affine torus, for a matrix $M\in \GL(n,\R)$.  A \emph{polarization} on $B(M)$ is equivalent to a symmetric positive-definite matrix $g$ such that $g \cdot M$ has integer entries (see Lemma \ref{lem:whenmpol}). Recall that an abelian group $A$ is \emph{divisible} if for each $a\in A$ and $n \in \Z_{>0}$ there exists $b \in A$ such that $n\cdot b = a$.

\begin{thm}[= Corollary \ref{cor:divisible}] \label{thm:divisible}
Suppose $B$ is a tropical affine torus which admits a polarization. Then the group $\Cob_{\fib}(X(B)/\C^*)_{\hom}$ is divisible.
\end{thm}

For the next result we modify the notion of Lagrangian brane to incorporate an additional piece of structure: an $\omega$-compatible almost-complex structure $J_L$ such that the Lagrangian bounds no $J_L$-holomorphic discs and intersects no $J_L$-holomorphic spheres. 
We call such $(L,J_L)$ an \emph{unobstructed} Lagrangian brane, and denote the corresponding cylindrical cobordism group by $\Cob^\unob(X/\C^*)$ (Section \ref{Sec:floer_cob}). 
A fibre $F_b$ of an integrable system (with no singular fibres) has this property, for any $\omega$-compatible $J$, and therefore admits a structure of unobstructed Lagrangian brane. 
The condition we call unobstructedness is often called  \emph{tautological} unobstructedness, to distinguish it from the more general notion of \emph{Floer-theoretic} unobstructedness.

A \emph{tropical $k$-cycle} in a tropical affine manifold $B$ is a weighted, balanced rational polyhedral complex of dimension $k$. A tropical cycle is called \emph{effective} if all weights are positive, an effective tropical cycle is also called a \emph{tropical subvariety}, and an effective tropical $1$-cycle is a \emph{tropical curve}.   Our primary non-existence theorem for cobordisms is:

\begin{thm}[see Proposition \ref{prop:nosepcob}]\label{thm:none}
Suppose $B$ contains no nonconstant tropical curve. Let $\bL^- = (F_{b_1^-},\ldots,F_{b_k^-})$ and $\bL^+ = (F_{b_1^+},\ldots,F_{b_k^+})$ be tuples of pairwise-distinct fibres in $X(B)$ equipped with their natural brane structures. Then there is no unobstructed cylindrical Lagrangian brane cobordism between $\bL^-$ and $\bL^+$, unless $\bL^-=\bL^+$ as unordered tuples.
\end{thm}

We give examples in which the hypothesis holds and the conclusion does not follow from purely topological arguments (i.e. a formal /  immersed Lagrangian brane cobordism exists).  We conjecture, \ref{conj:free}, that the hypothesis implies that $\Cob_{\fib}^\unob (X(B)/\C^*) \simeq \bZ^B$ is actually free.  For a concrete class of examples to  compare directly with Theorem \ref{thm:divisible}, we show that a sufficiently irrational tropical affine torus which admits no polarization contains no tropical curve. 

\begin{rmk}
\label{rmk:samegr}
We believe it should be possible to strengthen Theorems \ref{thm:divisible} and \ref{thm:none} by replacing the notion of `cobordism' in the former, and `unobstructed cobordism' in the latter, with the common notion of `Floer-theoretically unobstructed cobordism'. 
\end{rmk}


\subsection{Lagrangian cobordism and tropical Chow}

Our main results are based on an analogy between $\Cob_{\fib}(X(B)/\C^*)$ and the `tropical Chow group' $\CH_0(B)$. 
The latter is the quotient of the free abelian group generated by points $b \in B$ by `tropical linear equivalence', which is the equivalence relation arising from tropical curves in $\R \times B$ which are asymptotic to the given points near $\pm \infty \times B$ (see Section \ref{Sec:Chow} for details).

Naively, the idea is to try to construct comparison maps
\[ \xymatrix{ \CH_0(B)\ar@<.5ex>[rr] && \Cob_{\fib}^{\unob}(X(B)/\C^*)  \ar@<.3ex>[ll]  }\]
which are inverse to each other. 
However we are only able to construct modified versions of these maps, and indeed we do not expect such maps to exist in general. 
We summarize the maps that we are able to construct in Figure \ref{fig:equiv_relations}. 

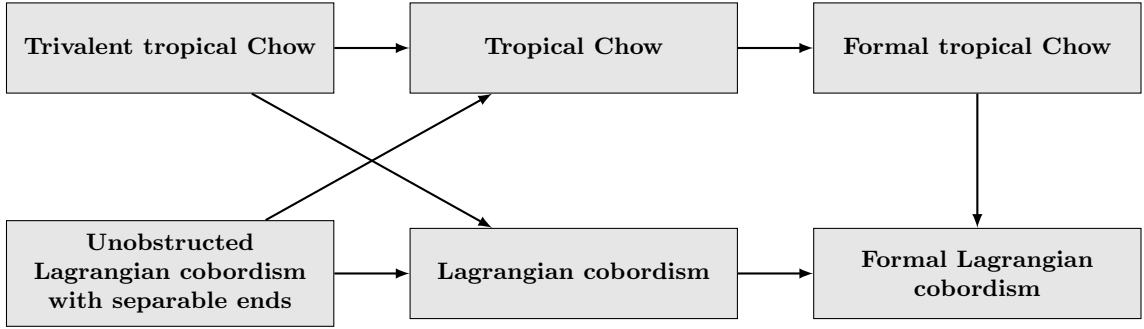
\begin{figure}[h]
\centering
\begin{tikzpicture}
[node distance = 1cm, auto,font=\footnotesize,
every node/.style={node distance=3cm},
comment/.style={rectangle, inner sep= 5pt, text width=4cm, node distance=0.25cm, font=\scriptsize\sffamily},
force/.style={rectangle, draw, fill=black!10, inner sep=5pt, text width=4cm, text badly centered, minimum height=1.2cm, font=\bfseries\footnotesize\sffamily}] 

\node [force] (cob) {Lagrangian cobordism};
\node [force, above of=cob] (trop) {Tropical Chow};
\node [force,  left=1cm of trop] (trivalent_trop) {Trivalent tropical Chow};
\node[force, right=1cm of trop] (formal_trop) {Formal tropical Chow};
\node [force, left=1cm of cob] (unobs_cob) {Unobstructed Lagrangian cobordism with separable ends};
\node [force, right=1cm of cob] (formal_cob) {Formal Lagrangian cobordism};

\path[->,thick] 
(trivalent_trop) edge (trop)
(trivalent_trop) edge (cob)
(trop) edge (formal_trop)
(formal_trop) edge (formal_cob)
(unobs_cob) edge (trop) 
(unobs_cob) edge (cob) 
(cob) edge (formal_cob);

\end{tikzpicture} 
\caption{Cobordism and tropical equivalence relations}
\label{fig:equiv_relations}
\end{figure}

The proof of Theorem \ref{thm:divisible} uses a version of the comparison map $\CH_0 \to \Cob$: by showing there are sufficiently many relations in the tropical Chow group to make it divisible, we conclude that the same must be true in the Lagrangian cobordism group. 
The construction of this comparison map involves building Lagrangian cobordisms from tropical curves, following \cite{Matessi2018,Mikhalkin2018, Mak-Ruddat}.
We only build such cobordisms for \emph{trivalent} tropical curves (see Figure \ref{fig:equiv_relations} and the explanation in Section \ref{Sec:CurvetoCob}), but this suffices for the proof of the theorem.

The proof of Theorem \ref{thm:none} uses the partial comparison map in the other direction, $\Cob \to \CH_0$: by showing there are no relations in the tropical Chow group, we conclude that the same must be true in the Lagrangian cobordism group.
We now describe the construction of this map, which involves building tropical curves in $\R \times B$ from Lagrangian cobordisms in $\C^* \times X(B)$. 

\begin{rmk}
We alert the reader that the construction uses a version of the result `the tropicalization of an analytic curve is a tropical curve' which is well-known to experts but has only appeared in the literature under extra hypotheses. See Section \ref{sec:nonArchSYZ} for details.
\end{rmk}

The mirror to $X(B)$ is a rigid analytic space defined over the universal one-variable Novikov field $\Lambda = \Bbbk\laurent{t^\R}$, where $\Bbbk$ is an algebraically closed field of characteristic $2$. 
The Novikov field comes with a non-Archimedean valuation $val: \Lambda^*\to \R$, and we let $U_{\Lambda} = val^{-1}(0)$ denote the unitary subgroup. 
As a set, the mirror rigid analytic space $Y(B)$ is the total space of the local system $T_\Z B \otimes_\Z U_\Lambda$; in particular it comes with a `non-Archimedean SYZ fibration' $val: Y(B) \to B$, locally modeled on the map $(\Lambda^*)^n \to \R^n$ taking valuations of all coordinates. 
If $B$ is a tropical affine torus, then $Y(B)$ is an analytic torus; and if $B$ is polarized, then $Y(B)$ is the analytification of an abelian variety \cite[Theorem 6.6.1]{FvdP}. 
The \emph{tropicalization} of an analytic subset of $Y(B)$ is its image in $B$, which is expected carry the structure of a tropical subvariety of the same dimension.

The mirror to $\C^*$ is the rigid analytic space $\mathbb{G}_m$ with underlying set $\Lambda^*$, and the mirror to $\C^* \times X(B)$ is $\mathbb{G}_m \times Y(B)$. 
Given an unobstructed Lagrangian brane cobordism $V\subset \C^*\times X(B)$, there is a corresponding ``family Floer sheaf" over $\mathbb{G}_m \times Y(B)$ \cite{Abouzaid:ICM}, whose support is an analytic subset $\scrE_V \subset \mathbb{G}_m \times Y(B)$. 
The tropical curve in $\R \times B$ associated to $V$ is the tropicalization of the union of one-dimensional irreducible components of $\scrE_V$. 
In order to show it has the desired form near infinity, we need to assume that $V$ has `separable ends' (see Figure \ref{fig:equiv_relations} and Section \ref{Sec:separable} for explanation).
This is satisfied \emph{a fortiori} in the setting of Theorem \ref{thm:none}, where the ends are assumed to be distinct.


\begin{rmk} \label{rmk:mirror_inexact}
Our results apply just as well when the canonical symplectic form $\omega_{can}$ on $X(B)$ is replaced with a symplectic form $\omega_{\alpha} = \omega_{can} + p^*\alpha$ for a closed 2-form $\alpha$ representing a non-zero  class in $H^2(B;\R)$. Then $(X(B),\omega_\alpha)$ admits the structure of a Lagrangian fibration but may have no Lagrangian section,  so a mirror ``space" will have no structure sheaf: it is a non-commutative deformation of $Y(B)$, realised in \cite{Abouzaid:ICM} by a gerbe on $Y(B)$. The non-commutative space still has a covering by affinoid domains, and one can define the analogue of a coherent sheaf and the tropicalization of its support. 
\end{rmk}

\subsection{Cobordism and rational equivalence}
\label{Sec:Cob_Rat_Eq}

We now explain the conjectural mirror relationship between cylindrical cobordism groups of fibres and rational equivalence of $0$-cycles, and how this motivates our results. 

For any analytic variety $Y$, a $0$-cycle is a finite formal linear combination of points of $Y$; two $0$-cycles $Z_0, Z_{\infty}$ are rationally equivalent if  there is a $1$-cycle $\Sigma \subset \bP^1 \times Y$ dominant over $\bP^1$ whose fibres over $0, \infty \in \bP^1$ (counted with multiplicity) are $Z_0, Z_\infty$. (More precisely, this geometric relationship generates the equivalence relation of rational equivalence.)  The Chow group $\CH_0(Y)$ is the group of rational equivalence classes of $0$-cycles on $Y$.

Now suppose that $X$ and $Y$ are homologically mirror.  
Using the language of Sylvan \cite{Sylvan}, the mirror to $\bP^1$ is the partially wrapped Fukaya category $\scrW(\C^*,\sigma)$, where the cylinder $\C^*$ is equipped with a single `stop' on each ideal boundary component. 
Therefore the mirror to $\bP^1 \times Y$ is $\scrW(\C^* \times X,\sigma \times X)$. 
A cobordism with one negative and one positive end can naturally be considered as an object of the latter category, and therefore is mirror to an object of $D^b\Coh(\bP^1 \times Y)$, which can be considered as a `family of objects of $D^b\Coh(Y)$ parametrized by $\bP^1$'.  
This is a hint that cylindrical cobordism should bear some mirror relation with rational equivalence, which is one of the main reasons we study cylindrical (rather than, e.g., planar) cobordism.\footnote{On the other hand, a planar cobordism with $n$ ends corresponds to an object of $\scrW_{\sigma \times X}(\C \times X)$ where $\sigma$ consists of $n$ stops in the ideal boundary of $\C$ -- this corresponds to a generalized $n$-triangle in the Fukaya category. So while there is reason to hope that planar or cylindrical cobordism might generate all relations in the Grothendieck group of the Fukaya category, planar cobordism does not reflect the geometry of rational equivalence on the mirror in the same way.}

More precisely, we expect that $\Cob^\unob(X/\C^*)$ should be mirror to $K(Y)$, the Grothendieck group of coherent sheaves on $Y$; and in accordance with the SYZ conjecture we expect that $\Cob^\unob_\fib(X(B)/\C^*)$ should be mirror to the subgroup of $K(Y(B))$ generated by coherent sheaves with $0$-dimensional support, which is isomorphic to $\CH_0(Y(B))$. 
We elaborate on this circle of ideas in Section \ref{Sec:cob_k}, and give further evidence for our expectations by constructing a comparison map between cylindrical Lagrangian cobordism and the Grothendieck group of the category of proper modules over the Fukaya category, cf. Proposition \ref{prop:cob_rel_groth}. 

Now we state the theorems about Chow groups which motivate our main results. 
Let $\CH_0(Y)_{\hom}$ denote the group of homologically trivial $0$-cycles, i.e. the kernel of the natural degree map
\[
deg: \CH_0(Y) \to \Z.
\]
The following result is well-known, but we reproduce the proof because the proofs of Theorems \ref{thm:divisible} and  \ref{thm:none} are modeled on it:

\begin{prop} \label{prop:mirror_statements} Let $Y$ be an analytic variety.
\begin{enumerate}
\item (see \cite[Lemma 1.3]{Bloch1976} or \cite[Exercise 1.6.6]{Fulton:intersection_theory}) If $Y$ is algebraic over an algebraically closed field, then $\CH_0(Y)_{\hom}$ is divisible.
\item If $Y$ contains no curve, then the relation of rational equivalence on $0$-cycles is trivial: so $\CH_0(Y) \simeq \Z^Y$.
\end{enumerate}
\end{prop}

\begin{proof}  The group $\CH_0(Y)_\hom$ is generated by classes $[p-q]$ where $p$ and $q$ are closed points of $Y$, so it suffices to show that such classes are $\ell$-divisible for any $\ell>0$. One can find a smooth curve $C$ and a morphism $f: C \to Y$ whose image contains both $p$ and $q$.  There is a push-forward map $\CH_*(C) \to \CH_*(Y)$. On the Jacobian $\Pic^0(C)$, the line bundle $\mathcal{O}(p-q)$ is $\ell$-divisible, so one can take a divisor $D_{\ell}$ on $C$ with the property that $\ell\cdot D_{\ell} \simeq (p-q)$. Pushing forward this relation shows that $[p-q] \in \CH_0(Y)_{\hom}$ is $\ell$-divisible.

For the second statement, any non-trivial relation between $0$-cycles would come from a curve on $\bP^1\times Y$. The projection of any component of this curve to $Y$ must be constant, hence the curve is a union of $\bP^1\times\{pt\}$ cycles, which implies that it yielded a trivial relation in the Chow group.
\end{proof} 

We remark that a general rigid analytic torus of dimension $n$ contains no analytic subset of dimension $\neq 0,n$, cf. \cite[Ex. 6.6.5]{FvdP}.

We are able to make the relation between cylindrical cobordism and rational equivalence precise under certain hypotheses. 
We prove that the support of the family Floer sheaf $\scrE_V$ compactifies to define an analytic subset of $(\bP^1)^{an}\times Y(B)$.  We deduce, cf. Corollary \ref{cor:cobordism_gives_rational_equiv}, that if $V\subset \C^*\times X(B)$ is an unobstructed Lagrangian cobordism with distinct ends, i.e. between tuples of pairwise distinct  Lagrangian fibres $(F_{b_1^-},\ldots, F_{b_k^-})$ and $(F_{b^+_1}, \ldots, F_{b_k^+})$, then the effective, reduced $0$-cycles $\sum_{j=1}^k [b^-_j]$ and $\sum_{j=1}^k [b^+_j]$ on $Y(B)$ are rationally equivalent. 
 In particular, the northeast arrow in Figure \ref{fig:equiv_relations} factors through the Chow group $\CH_0(Y(B))$ (strictly speaking, we have only proved this if `separable ends' is replaced by `distinct ends').

We take up this circle of ideas in the context of K3 surfaces in a separate paper \cite{SS:K3_spherical_ring}. 
\vspace{1em}

\paragraph{\textbf{Acknowledgements.}} Theorem \ref{thm:none} realises a suggestion of Paul Seidel, who also suggested considering cylindrical rather than planar cobordism.  We are grateful to Mohammed Abouzaid, Denis Auroux, Jeff Hicks, Daniel Huybrechts, Joe Rabinoff, Dhruv Ranganathan, Tony Scholl, Paul Seidel, Harold Williams and Tony Yue Yu for helpful conversations. 

N.S. was partially supported by a Royal Society University Research Fellowship.  I.S. was partially supported by a Fellowship from E.P.S.R.C. and by the National Science Foundation under Grant No. DMS-1440140, whilst in residence at the Mathematical Sciences Research Institute in Berkeley, California, during the Spring 2018 semester.

\section{Cobordism preliminaries}

\subsection{Tropical affine manifolds}

Recall that the group $\R^n \rtimes \GL(n,\Z)$ acts on $\R^n$ by affine automorphisms, $m \mapsto Am + b$. 

\begin{defn}
A \emph{tropical affine manifold} is a smooth manifold equipped with an atlas with transition functions in $\R^n \rtimes \GL(n,\Z)$.
\end{defn}

\begin{example}
\label{eg:tori}
Associated to a matrix $M \in \GL(n,\R)$ we have a tropical affine manifold $B(M) := \R^n/(M \cdot \Z^n)$. 
We call such a manifold a \emph{tropical affine torus}. 
\end{example}

A tropical affine manifold $B$ comes equipped with a subsheaf $\Aff_B \subset C^\infty_B$ of the sheaf of smooth functions on $B$, consisting of those functions which are affine-linear with integer slope in any coordinate chart; 
 a sub-bundle $T^*_\Z B\subset T^* B$ consisting of the differentials of sections of $\Aff_B$; and a sub-bundle $T_\Z B \subset TB$ consisting of those tangent vectors on which all elements of $T^*_\Z B$ evaluate to integers. 
The latter two are local systems of free abelian groups of rank $n$.

\begin{defn}
A \emph{tropical affine map} between tropical affine manifolds $B_1$ and $B_2$ is a smooth map $B_1 \to B_2$ which is affine with integer linear part in any choice of affine charts.
\end{defn}

Note that a map $f: B_1 \to B_2$ is tropical affine if and only if $f^* (\Aff_{B_2}) \subset \Aff_{B_1}$.

\begin{lem}
Tropical affine tori $B(M)$ and $B(N)$ are isomorphic if and only if $[M]=[N]$ in $\GL(n,\Z) \backslash \GL(n,\R) / \GL(n,\Z)$.
\end{lem}
\begin{proof}
If $M=ANB$ for $A,B \in \GL(n,\Z)$, then the map
\begin{align*}
\R^n/(M \cdot \Z^n) &\to \R^n/(N \cdot \Z^n) \\
x & \mapsto A^{-1}\cdot x
\end{align*}
is well-defined, because $A^{-1}\cdot (M\cdot \Z^n) = N\cdot (B \cdot \Z^n) = N \cdot \Z^n$; and it is a tropical affine map because $A^{-1}$ has integer entries. 
The converse is clear.
\end{proof}

\subsection{Albanese and Picard varieties}

Let $B$ be a tropical affine manifold. 
Given a one-chain $\gamma$ in $B$, there is an integration map
\begin{equation}
\label{eqn:intmap}
 \int_\gamma: H^0(T^*_\Z B) \to \R.
\end{equation}
In the case that $\gamma$ is closed, we obtain a pairing
\[ \int: H^0(T^*_\Z B) \otimes H_1(B;\Z) \to \R,\]
which is analogous to the integration pairing between holomorphic differential 1-forms and 1-cycles on a complex manifold. 
We define the \emph{tropical Albanese variety}
\[ \Alb(B) := \Hom(H^0(T^*_\Z B),\R) / H_1(B;\Z), \]
and the \emph{tropical Picard variety}\footnote{Following \cite[Section 5]{Mikhalkin2008}, the Picard variety is the connected component of the identity in the Picard group $\Pic(B) := H^1(\Aff_B)$, which is the kernel of the `tropical Chern character map' $H^1(\Aff_B) \to H^1(T^*_\Z B)$, as one sees using the long exact sequence of sheaf cohomology groups associated to the short exact sequence of sheaves $0 \to \R \to \Aff \to T^*_\Z \to 0$.}
\[ \Pic^0(B) := \Hom(H_1(B;\Z),\R)/H^0(T^*_\Z B).\]

In general it is not clear that these topological spaces will even be manifolds. 
However when the integration pairing is a perfect pairing (after tensoring with $\R$), they come equipped with natural structures of tropical affine tori. 
This is the case, for example, when $B$ is itself a tropical affine torus.

Let $C_0(B)$ denote the free abelian group generated by points of $B$, and $\tilde{C}_0(B)$ the kernel of the map $C_0(B) \to \Z$ summing the coefficients of the points. 
Then we have a well-defined map
\begin{align*}
\alb: \tilde{C}_0(B) & \to \Alb(B) \\
\alb(a)(\alpha) & := \int_\gamma \alpha, \text{ where $\partial \gamma = a$.}
\end{align*}

In particular, given a basepoint $b_0 \in B$, we obtain a tropical affine map
\begin{align*}
\alb:B & \to \Alb(B) \\
\alb(b) &:= \alb(b-b_0).
\end{align*}
called the \emph{tropical Albanese map}. 

\begin{rmk}
If $B$ is a tropical affine torus, then the Albanese map defines an isomorphism $B \simeq \Alb(B)$. 
On the other hand, if $B=B(M)$ then we have $\Pic^0(B) \simeq B(M^T)$.
\end{rmk}

We observe that any tropical affine map $f: B_1 \to B_2$ induces tropical affine maps
\begin{align*}
f_*: \Alb(B_1) & \to \Alb(B_2), \\
f^*: \Pic^0(B_2) & \to \Pic^0(B_1),
\end{align*}
and $f_*$ respects the tropical Albanese map in the obvious sense.

\subsection{Lagrangian torus fibrations}

Let $p: (X,\omega) \to B$ be a completely integrable system, i.e. a smooth Lagrangian torus fibration. 
We then have action-angle co-ordinates on $X$, and the action co-ordinates on the base $B$ endow it with the structure of a tropical affine manifold. 


Let us recall how the action co-ordinates are constructed. 
We define $F_b := p^{-1}(b)$ to be the torus fibre over $b \in B$.  
For any path $\gamma: [0,1] \to B$, we have the \emph{flux homomorphism}
\[ \varphi_\gamma: H_1(F_{\gamma(0)};\Z) \to \R\]
given by integrating $\omega_{can}$ over the two-cycle `swept out' by a one-cycle in $F_{\gamma(0)}$ over $\gamma$. 
Keeping $b_0 = \gamma(0)$ fixed and allowing $b = \gamma_b(1)$ to vary, we obtain a function $f_a(b) = \varphi_{\gamma_b}(a)$ for each $a \in H_1(F_{b_0};\Z)$. 
We declare $\Aff_B \subset C^\infty_B$ to be the subsheaf of all such functions $f_a$: this defines the tropical affine structure on $B$.

We observe that there is an identification $H_1(F_b;\Z) \simeq T^*_{\Z,b}$ under which $a$ corresponds to $df_a$. 
The flux homomorphism $\varphi_\gamma$ corresponds, via this identification, to the integration map 
\[ \int_\gamma: T^*_{\Z,\gamma(0)} \to \R\]
(which extends the covector over $\gamma(0)$ to one over $\gamma$ by parallel transport, then integrates it).

Let us now suppose that $B$ is a tropical affine \emph{torus}, so that the fibration $p: X(B) \to B$ is trivial. 
This means we can identify $H_1(F_b;\Z) \simeq H_1(F;\Z)$ for all $b$ without ambiguity; so for any one-cycle $\gamma$ in $B$ we obtain a flux homomorphism
\[ \varphi_\gamma: H_1(F;\Z) \to \R.\]
This coincides, via the identification $H_1(F;\Z) \simeq H^0(T^*_\Z B)$, with the integration map \eqref{eqn:intmap}. 


\subsection{Lagrangian cobordism\label{Sec:lag_cobordism}} Let $(X,\omega)$ be a closed symplectic manifold equipped with an $\infty$-fold Maslov cover $\mathcal{L}$ of the oriented Lagrangian Grassmannian (in the sense of \cite{Seidel:graded}). 
A \emph{Lagrangian brane} is a Lagrangian submanifold $L \subset X$ equipped with a grading (i.e., a lift to $\mathcal{L}$). 
Fix tuples $\bL^- = (L_1^-,\ldots, L_r^-)$ and $\bL^+ = (L_1^+, \ldots, L_s^+)$ of Lagrangian branes; these need not be pairwise distinct either within or between the two collections.  A \emph{planar cobordism} between $\bL^-$ and $\bL^+$ is a Lagrangian brane
\[
V \subset (\C \times X, \omega_\C \,\oplus \omega)
\]
which has cylindrical ends: outside a compact subset, it co-incides with the union of  product Lagrangian branes
\[
\bigsqcup_{i=1}^r \, (-\infty, -1] \times \{x^-_i\} \times L_i^- \ \sqcup \ \bigsqcup_{j=1}^s \, [1,\infty) \times \{x^+_j\} \times L_j^+
\]
for pairwise distinct real numbers $\{x_i^-\}$ and pairwise distinct real numbers $\{x_j^+\}$. We will refer to the $\{x_i^{\pm}\}$ as the negative and positive heights of the cobordism. (By convention, positive heights are ordered by increasing $x_j^+$ whilst negative heights are ordered by decreasing $x_j^-$, to reflect the direction of the Reeb flow on the ideal circle boundary of $\C$.)  The projection of $V$ to $\C$ co-incides with a collection of parallel horizontal lines near each of positive and negative real infinity. We will call these the \emph{ends} of $V$, and will say the  Lagrangian brane $L_i^-$ or $L_j^+$ is carried by the corresponding end.  If a single brane $L\subset X$ is carried by $k$ distinct (positive resp. negative) ends, we say that $L$ has \emph{multiplicity} $k$ (at positive resp. negative infinity).

\begin{rmk} Planar Lagrangian cobordism was introduced by Arnol'd \cite{Arnold1, Arnold2} and has been studied by a number of authors: non-exhaustively, \cite{Audin} treats the underlying topological aspects; \cite{Biran-Cornea, Biran-Cornea-2} study Floer theoretic rigidity;  \cite{Haug} computes the cobordism group of $T^2$.  There is a related but distinct notion of ``Lagrangian cobordism between Legendrian submanifolds of a contact manifold", also  much studied -- see e.g. \cite{Chantraine, BST} and references therein -- which in general yields an asymmetric relation on Legendrians, and which is not our focus here.
\end{rmk}

Given the same initial data $(X, \bL^-, \bL^+)$, a \emph{cylindrical cobordism} between $\bL^-$ and $\bL^+$ is a Lagrangian brane
\[
V \subset (\C^* \times X, \omega_{\C^*} \, \oplus \omega)
\]
which should be asymptotically cylindrical, so its projection to $\C^*$ should  co-incide outside a compact set with a union of radial lines near both $0$ and $\infty$ in $\C^*$, carrying $\bL^-$ and $\bL^+$ respectively, cf. Figure \ref{Fig:cylindrical_cobordism}.   (We fix the co-ordinate $z=re^{i\theta}$ on $\C^*$, the symplectic form $\omega_{\C^*} = dr\wedge d\theta$, and the $\infty$-fold Maslov covering of $\C^*$ defined by the volume form $dz$; equivalently, we view $\C^*=X(\R)$ and follow the general conventions of Section \ref{Sec:AlbCob}. The brane structures for cylindrical cobordisms are then defined with respect to product data.)

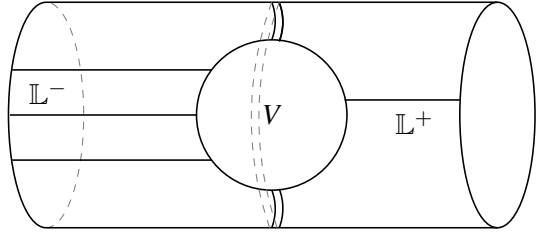
\begin{figure}[ht]
\begin{center} 
\begin{tikzpicture}
	\draw[dashed,color=gray] (0,0) arc (-90:90:0.5 and 1.5);
	\draw[semithick] (0,0) -- (6,0);
	\draw[semithick] (0,3) -- (6,3);
	\draw[semithick] (0,0) arc (270:90:0.5 and 1.5);
	\draw[semithick] (6,1.5) ellipse (0.5 and 1.5);
	\draw (0.03,1.8) node {$\mathbb{L}^-$};
	\draw (4.9,1.4) node {$\mathbb{L}^+$};
	\draw[semithick] (-0.45,2.1) -- (2.2,2.1);
	\draw[semithick] (-0.48,1.5) -- (2,1.5);
	\draw[semithick] (-0.45,0.9) -- (2.2,0.9);
	\draw[semithick] (3.97,1.7) -- (5.5,1.7);
	\draw[semithick] (3,1.5) circle(1cm) node{$V$};
	\draw[semithick]  (3,2.5) arc (-21:21:0.7) ;
	\draw[semithick]  (3.1,2.48) arc (-21:21:0.72) ;
	\draw[semithick]  (3.1,2.48) arc (-21:21:0.72) ;
	\draw[semithick]  (3.1,0) arc (-21:21:0.72);
	\draw[semithick]  (3,0) arc (-21:21:0.7) ;
	\draw[dashed,color=gray] (3.1,3) arc (160:200:4.5) ;

	\draw[dashed,color=gray] (3,3) arc (160:200:4.5) ;

\end{tikzpicture}
\end{center}
\caption{A cylindrical cobordism with three negative ends and one positive end.}
\label{Fig:cylindrical_cobordism}
\end{figure}


\begin{example}
For any Lagrangian brane $L \subset X$, we have $L[1] = -L$ in the (cylindrical or planar) Lagrangian cobordism group. 
This relation is realized by the cobordism $V = \gamma \times L$ where $\gamma$ is a path with two negative and no positive ends. 
As a consequence we have $L[2] = L$.
\end{example}

\begin{example}
For disjoint Lagrangian branes $L_1$ and $L_2$ in $X$, we have $L_1 \sqcup L_2 = L_1 + L_2$ in the (cylindrical or planar) Lagrangian cobordism group. 
This relation is realized by the cobordism $V = \gamma_1 \times L_1 \sqcup \gamma_2 \times L_2$, where $\gamma_i$ are paths that are parallel and distinct near negative real infinity, and coincide near positive real infinity.
\end{example}

Any planar cobordism gives rise to a cylindrical one by quotienting $\C$ by a large imaginary translation. Therefore we have a well-defined map
\[ \Cob(X/\C) \to \Cob(X/\C^*), \]
but not one in the reverse direction because $\Cob(X/\C^*)$ has additional relations from cobordisms which wrap the cylinder non-trivially.

Finally, it is sometimes enlightening to enhance our notion of Lagrangian brane to include a $U_\Lambda$-local system. 
All the definitions stay the same: in particular cobordisms should come equipped with local systems as part of their brane structure, which restrict to the local systems over the ends. 
We will denote the resulting version of cylindrical cobordism by $\Cob^{\loc}(X/\C^*)$. 
One reason this is a useful thing to do is that the points of the rigid analytic space $Y(B)$ mirror to a Lagrangian torus fibration $X(B)$ correspond precisely to fibres of $X(B)$ equipped with $U_\Lambda$-local systems.

\subsection{The Albanese map on cobordism}\label{Sec:AlbCob}

Let $B$ be an oriented tropical affine manifold. 
The orientation on $B$ determines a canonical $\infty$-fold Maslov cover $\mathcal{L}$ of the oriented Lagrangian Grassmannian of $X = X(B)$, which we now describe. 
For any $\omega$-compatible almost-complex structure $J$ on $X$, there is a canonical choice of $J$-holomorphic volume form $\eta_J \in \Omega^{n,0}(X)$: if $b_1,\ldots,b_n$ are oriented local affine coordinates on $B$, we define
\begin{equation}
\label{eqn:flatvolform}
\eta_J := (d^c b_1 + idb_1) \wedge \ldots \wedge (d^c b_n + idb_n).
\end{equation}
If $V \subset TX$ is an oriented Lagrangian subspace, we define $\theta_V \in S^1$ by $\eta_J|_V = e^{i\theta_V} \cdot vol$ where $vol$ is an oriented volume form on $V$. 
A grading for $V$ is a lift of $\theta_V$ to $\theta^\#_V \in \R$. 
The space of all graded oriented Lagrangian subspaces forms the desired Maslov cover $\mathcal{L}$.
Since the space of $\omega$-compatible almost-complex structures is contractible, it does not matter which $J$ we choose. 

If $i:L \to X$ is an oriented Lagrangian immersion, we obtain a \emph{phase function} $\theta_L: L \to S^1$ by $ i^* \eta_J = e^{i \theta_L} \cdot vol$, 
where $vol \in \Omega^n(L)$ is an oriented volume form. 
A grading for $L$ is a lift of $\theta_L$ to $\theta_L^\#:L \to \R$. 
A Lagrangian brane is equivalent to an oriented Lagrangian submanifold equipped with a grading.  
If the phase function $\theta_L \in \R/2\pi\Z$ is constant, we say $L$ is \emph{special Lagrangian of phase $\theta_L$} with respect to $J$. 

The orientation of $B$ endows a Lagrangian torus fibre $F_b$ with a natural orientation via the isomorphism $H_1(F_b;\Z) \simeq T^*_\Z B$. 
Examination of \eqref{eqn:flatvolform} shows that $F_b$, equipped with its natural orientation, is special Lagrangian of phase $0$ with respect to any $\eta_J$. 
Thus $F_b$ admits a natural brane structure: we will abuse notation by denoting this Lagrangian brane also by $F_b$.

\begin{defn}
We define $\Cob_{\fib}(X/\C^*)$ to be the free abelian group generated by Lagrangian torus fibres equipped with brane structures, modulo the equivalence relation generated by Lagrangian brane cobordisms, all of whose ends are fibres.
\end{defn}

Note that there is a map $\Cob_{\fib}(X/\C^*) \to \Cob(X/\C^*)$, which is surjective onto the subgroup generated by fibres, however it need not be an isomorphism onto this subgroup. 

We recall that $\Cob_{\fib}(X/\C^*)_{\hom} \subset \Cob_{\fib}(X/\C^*)$ denotes the homologically-trivial subgroup.

\begin{lem}
\label{lem:albwelldef}
There is a well-defined map 
\begin{align*}
\alb: \Cob_{\fib}(X/\C^*)_\hom & \to \Alb(B), \text{ sending} \\
[F_b] - [F_a] & \mapsto \alb(b-a).
\end{align*}
\end{lem}
\begin{proof}
Suppose $V \subset \C^*\times X $ is a Lagrangian brane cobordism with ends $\bL^+ = \{F_{b_i}, i=1,\ldots,k\}$ and $\bL^- = \{F_{a_i}, i=1,\ldots,k\}$.  
We must show that $\alb\left(\sum_i (b_i - a_i)\right) = 0$.

By altering the ends of $V$ we can arrange that they all terminate on a single fibre $\{z\}\times X$ (at the cost of making $V$ immersed rather than embedded). 
Let $\gamma$ be a one-chain in $B$ with $\partial \gamma = \sum_i (b_i -  a_i)$. 
Let $W = \{z\} \times \gamma \times F$ be the corresponding $(n+1)$-chain in $\C^*\times X$. 
Putting it together with $V$ we get a closed $(n+1)$-cycle $V + W$ in $\C^*\times X$. 

Let us denote by $\mathrm{pr}_{p,q}: H_{p+q}(X) \to H_p(B) \otimes H_q(F)$
the projection arising from the K\"unneth decomposition of the homology of $X \simeq B \times F$.
We define $\beta \in H_1(B)$ to be the class such that $\mathrm{pr}_{1,n} (\pi_X)_* (V + W) = \beta \otimes [F]$.
We will prove that
\[ \alb\left(\sum_i b_i - \sum_i a_i\right) = \varphi_\beta,\]
from which it follows that the image in $\Alb(B)$ vanishes.

To prove this, let $\alpha$ be a one-cycle representing a homology class in $H_1(F)$. 
Consider the two-cycle $u = (V + W) \pitchfork (\C^*\times [B] \times \alpha)$ (perturbing to make the intersection transverse). 
We decompose $u$ into chains $u_V \subset V$ and $u_W \subset W$ having a common boundary. 
Now observe that 
\begingroup
\allowdisplaybreaks
\begin{align*}
\alb\left(\sum_i b_i - \sum_i a_i\right)(\alpha) &= \varphi_\gamma(\alpha) \\
&= \int_{u_W} \omega_X \quad \text{(since $u_W$ is the two-cycle swept by $\alpha$ over $\gamma$)}\\
&= \int_{u_W} \omega_{\C^*} \oplus \omega_{X} \quad\text{(since $\pi_{\C^*}(W) = \{z\}$)} \\
&= \int_u \omega_{\C^*} \oplus \omega_{X}  \quad \text{(since $\omega_{\C^*} \oplus \omega_{X}|_V = 0$)} \\
&= \int_u \omega_X 
 \quad \text{(since $\omega_{\C^*}$ is exact and $u$ is closed)} \\
&=\int_{(\pi_{X})_* ((V + W) \pitchfork \C^*\times [B] \times \alpha)} \omega_X \quad \text{(by definition of $u$)} \\
&= \int_{(\pi_X)_* (V + W) \pitchfork [B] \times \alpha} \omega_X \\
&= \int_{\beta \times \alpha} \omega_X \\
&= \varphi_\beta(\alpha).
\end{align*}
\endgroup
The penultimate line follows because $\omega_X$ vanishes on $H_2(B) \otimes H_0(F)$ (because the zero-section $B$ is Lagrangian) and on $H_0(B) \otimes H_2(F)$ (because the fibre $F$ is Lagrangian), so the integral only picks up the component in $H_1(B) \otimes H_1(F)$. 
It is clear that if $\mathrm{pr}_{1,n}(\eta) = \beta \otimes [F]$, then $\mathrm{pr}_{1,1}(\eta \pitchfork [B] \times \alpha) = \beta \otimes \alpha$,
and the claim follows.
\end{proof}

\begin{rmk}
There is an analogue of the Albanese map incorporating $U_\Lambda$-local systems, which fits into a commutative diagram
\[ \xymatrix{ \Cob_{\fib}^{\loc}(X/\C^*)_{\hom} \ar[r]^-{\alb^{\loc}} \ar[d] & Y(\Alb(B)) \ar[d] \\
\Cob_{\fib}(X/\C^*)_\hom \ar[r]^-{\alb} & \Alb(B),} \]
where the left vertical map forgets the local system and the right vertical map is a non-Archimedean SYZ fibration (see Section \ref{sec:nonArchSYZ}).
\end{rmk}

\section{Tropical curves}

\subsection{Chow group\label{Sec:Chow}}

If $B$ is a compact tropical affine manifold, we recall that a \emph{tropical $k$-cycle} $V \subset B$ is a weighted balanced rational polyhedral complex of pure dimension $k$ (see \cite[Section 4]{Mikhalkin2006},  \cite{Allermann2010,Mikhalkin2015}).\footnote{We will assume that our polyhedral complexes are locally, but not necessarily globally, finite.} 
The cycle is called \emph{effective} if all weights are positive; an effective tropical $k$-cycle is also called a $k$-dimensional \emph{tropical subvariety}. 

Two tropical $k$-cycles $V_+,V_- \subset B$ are called \emph{linearly equivalent} if there exists a tropical $(k+1)$-cycle $V \subset B \times \R$ which coincides with $V_\pm \times \R$ over some neighbourhood of $\pm \infty$ in the $\R$ factor. 
We will say that two $k$-dimensional tropical subvarieties are \emph{effectively linearly equivalent} if the same holds with a $(k+1)$-dimensional tropical subvariety $V$. 

\begin{defn}
\label{defn:tropchow}
We define the tropical Chow group $\CH_k(B)$ to be the free abelian group generated by $k$-dimensional tropical subvarieties in $B$ modulo effective linear equivalence.
\end{defn}

\begin{rmk}
\label{rmk:weirdchow}
The tropical Chow group is usually defined in the literature to be the free abelian group generated by tropical $k$-cycles modulo linear equivalence \cite{Mikhalkin2006,Allermann2010,Mikhalkin2015,Maclagan2007}. 
These two groups in general do not coincide (see Remark \ref{rmk:diffchow}); it is the version in Definition \ref{defn:tropchow} that is most relevant to our paper. 
\end{rmk}

We define the \emph{degree} of a tropical $0$-cycle to be $\deg\left(\sum_p n_p \cdot [p]\right) := \sum_p n_p$. 
It is straightforward to check that this gives a well-defined map $\deg:\CH_0(B) \to \Z$, as a consequence of the balancing condition.
We denote its kernel by $\CH_0(B)_\hom$.

\subsection{Parametrized tropical curves}\label{subsec:paramcurve}

We will be most interested in \emph{tropical curves in $B$}, which are one-dimensional tropical subvarieties $V \subset B$. 
A tropical curve is equivalent to a graph (possibly with semi-infinite edges) which is properly embedded in $B$, so that each edge $E$ is a tropical submanifold of $B$. 
Each edge $E$ is equipped with a weight $w_E \in \N$, and the balancing condition $ \sum_{E_i \to v} w_{E_i} \cdot u_i = 0$ 
is satisfied at each vertex $v$: here the sum is over all edge-ends incident to $v$, and $u_i$ is the generator of $T_\Z E_i$ at $v$ pointing `into' $E_i$. 

We now recall the notions of abstract and parametrized tropical curves (see \cite{Mikhalkin2005}). 

\begin{defn}
An \emph{abstract tropical curve} $C$ is a graph with semi-infinite edges, each vertex of which has finite valency, equipped with a tropical affine structure on each edge. 
The semi-infinite edges are required to have infinite `length'.
\end{defn}

Note that a tropical affine structure on a one-manifold determines a Riemannian metric $dx^2$, where $dx$ is a generator of $T^*_\Z$; this is the notion of `length' referred to in the definition.

Sometimes we will drop the `abstract', when we feel no confusion can result. 

\begin{defn}
A \emph{parametrized tropical curve} in $B$ is an abstract tropical curve $C$ equipped with a continuous map $h: C \to B$ whose restriction to each edge $E$ is tropical affine, satisfying the balancing condition $\sum_{E_i \to v} dh(z_i) = 0$
at each vertex $v$ of $C$, where $z_i$ is the generator of $T_\Z E_i$ at $v$ pointing `into' $E_i$, and the sum is over all edge-ends incident to $v$.
\end{defn}

The image of a parametrized tropical curve is a tropical curve in $B$: the weight assigned to $E$ is the unique positive integer $w_E$ such that $dh(z_E) = w_E \cdot u_E$, where $z_E$ is a generator of $T_\Z E$ and $u_E$ is a primitive vector in $T_\Z B$.
Conversely, any tropical curve in $B$ admits a canonical parametrization.

Associated to an abstract tropical curve $C$ we have a subsheaf $\Aff_C \subset C^0(C)$ of the sheaf of continuous functions, consisting of all functions $f$ whose restriction to each edge $E$ lies in $\Aff_E$ and which satisfy $\sum_{E_i \to v} df(z_i) = 0$ at each vertex. 
We observe that a continuous map $h:C \to B$ determines a parametrized tropical curve if and only if $h^*(\Aff_B) \subset \Aff_C$.
There is also a sheaf $T^*_\Z C$, which is the quotient of $\Aff_C$ by the constant functions.
%

The tropical Chow group also makes sense for an abstract tropical curve $C$. 
A tropical divisor in $C$ is simply a finite integer linear combination $\sum_p n_p \cdot p$ of points $p \in C$; and this tropical divisor is called \emph{principal} if there exists a continuous function $f$ on $C$, affine away from the points $p$ at which $n_p \neq 0$, and which `increases slope by $n_p$' as we cross a point $p$ (see \cite[Section 4]{Mikhalkin2008} for the precise definition). 
We define $\CH_0(C)$ as the quotient of the group of tropical divisors by the principal ones. 

If $\sum_p n_p \cdot p$ is a principal tropical divisor associated to a continuous function $f$ and $h: C \to B$ is a parametrized tropical curve, then the tropical 0-cycle $\sum_p n_p \cdot f(p)$ vanishes in $\CH_0(B)$. 
Indeed, one can use $f$ and $h$ to construct an effective tropical cycle in $B \times \R$ which realizes this relation. 
This construction is called the  the `tropical graph cobordism' in \cite[Construction 3.3]{Allermann2010}, and the `full graph' in \cite{Mikhalkin2006}. 
The upshot is that there is a pushforward map $h_*: \CH_0(C) \to \CH_0(B)$.

%

\subsection{Jacobian variety}

We now recall the definition of the Jacobian variety of an abstract tropical curve, as well as the tropical Abel--Jacobi map, following \cite[Section 6]{Mikhalkin2008}.

The integration pairing makes perfect sense on an abstract tropical curve $C$, and in fact is perfect (after tensoring with $\R$): thus we can define tropical affine tori $\Alb(C)$ and $\Pic^0(C)$ as above. 
Furthermore, since a parametrized tropical curve $h: C \to B$ determines maps $h^*: H^0(T^*_\Z B) \to H^0(T^*_\Z C)$ and $h_*: H_1(C;\Z) \to H_1(B;\Z)$, we also have induced maps $h_*$ and $h^*$ on the Albanese and Picard varieties as above; and we have the Albanese map $\alb: C \to \Alb(C)$, which is a parametrized tropical curve. 

\begin{lem}[= {\cite[Lemma 6.1]{Mikhalkin2008}}]
\label{lem:jacpol}
If $C$ is a compact tropical curve, then there is a natural identification $H_1(C;\Z) \simeq H^0(T^*_\Z C)$ (a tropical version of Serre duality) which makes the integration pairing symmetric and positive definite. 
\end{lem}
\begin{proof}
Observe that, after choosing an orientation of the edges of $C$, both groups can be identified with
\[ \left\{ \vec{a} \in \Z^{\mathrm{Edges}(C)}: \sum_{E \to v} a_E = 0\right\}.\]
In the case of $H^0(T^*_\Z C)$, the basis element $\mathsf{e}_E$ corresponds to the generator of $T^*_\Z E$ with the appropriate orientation, and the constraint corresponds to the tropical balancing condition; in the case of $H_1(C;\Z)$, the basis element corresponds to the edge $E$ with the appropriate orientation, and the constraint corresponds to the condition that the cycle have vanishing boundary. 
The integration pairing extends to the larger group $\Z^{\mathrm{Edges}(C)}$, where its matrix is diagonal with entries equal to the lengths of the edges. 
In particular it is symmetric and positive-definite, and therefore restricts to such a pairing on the subgroup.
\end{proof}

The first consequence of Lemma \ref{lem:jacpol} is that the Albanese and Picard varieties of $C$ are canonically isomorphic. 

\begin{defn}
If $C$ is a compact tropical curve, we define its \emph{Jacobian} $\Jac(C) := \Alb(C) \simeq \Pic^0(C)$, and the \emph{Abel--Jacobi map} $\AJ: C \to \Jac(C)$ to correspond to the Albanese map.
\end{defn}

The Abel--Jacobi map extends to an isomorphism
\begin{equation}
\label{eqn:tropAJch}
 \AJ: \CH_0(C)_\hom \xrightarrow{\sim} \Jac(C)
\end{equation}
by the tropical Abel--Jacobi theorem \cite[Theorem 6.2]{Mikhalkin2008}. 

\begin{lem}
If $B$ is a tropical affine torus, then the tropical Albanese map extends to a map
\begin{equation}
\label{eqn:tropAlbch}
\alb: \CH_0(B)_\hom \to \Alb(B).
\end{equation}
If $h: C \to B$ is a parametrized tropical curve, then the diagram
\[ \xymatrix{ \CH_0(C)_\hom \ar[r]^-{\AJ} \ar[d]^-{h_*} & \Jac(C) \ar[d]^-{h_*} \\
\CH_0(B)_\hom \ar[r]^-{\alb} & \Alb(B)}\]
commutes.
\end{lem}
\begin{proof}
The main non-trivial step is to check that the map \eqref{eqn:tropAlbch} is well-defined. 
This follows from the fact that any linear equivalence in $B$ is induced by a parametrized tropical curve $h: C \to B$ via pushforward, together with the fact that the map \eqref{eqn:tropAJch} is well-defined. 
\end{proof}

\subsection{Polarizations}

We continue our review of the tropical Jacobian variety from \cite{Mikhalkin2008}.
Lemma \ref{lem:jacpol} implies that $\Jac(C)$ admits a canonical \emph{principal polarization} in the following sense:

\begin{defn}
\label{defn:pol}
A \emph{polarization} of a tropical affine torus is a map
\[ c: H_1(B;\Z) \to H^0(T^*_\Z B)\]
such that $\int_{\gamma_1} c(\gamma_2)$ defines a symmetric, positive-definite pairing on $H_1(B;\Z)$.\footnote{
We remark that Mikhalkin--Zharkov interpret this map as a class in
\[ \Hom(H_1(B;\Z),H^0(T^*_\Z B)) \simeq H^1(T^*_\Z B) \]
(observing that $T^*_\Z B$ is trivial on a tropical affine torus), and show that the resulting pairing on $H_1(B;\Z)$ is symmetric if and only if this class lies in the image of the tropical Chern character map $H^1(\Aff_B) \to H^1(T^*_\Z B)$.} 
The \emph{index} of the polarization is the index of the map; if the index is $1$ the polarization is called \emph{principal}.
\end{defn}

Let us denote the integration pairing by 
\begin{align*}
\ints:H_1(B;\Z) & \to H^0(T_\R B), \\
\int_\gamma \alpha &= \alpha(\ints(\gamma)),\qquad \text{
for any $\alpha \in H^0(T^*_\R B)$.}
\end{align*}
Recall that a Riemannian metric on $B$ induces an isomorphism $\flat: TB \xrightarrow{\sim} T^*B$. 

\begin{lem}\label{lem:whenmpol}
A polarization of the tropical affine torus $B$ is equivalent to a flat Riemannian metric on $B$, such that the map
\[ H_1(B;\Z) \xrightarrow{\ints} H^0(T_\R B) \xrightarrow{\flat} H^0(T^*_\R B)\]
lands in $H^0(T^*_\Z B)$.

If $B=B(M)$, then this is in turn equivalent to a symmetric positive-definite matrix $g$ such that $g \cdot M$ has integer entries.
\end{lem}
\begin{proof}
The polarization $c$ is related to the flat Riemannian metric by $c(\gamma) = \ints(\gamma)^\flat$. 

We can identify $\ints:H_1(B(M);\Z) \to H^0(T_\R B(M))$ with the map $\Z^n \to \R^n: x \mapsto M\cdot x$; then the matrix $g$ is the matrix of $\flat:\R^n \to (\R^n)^*$, in the standard basis.
\end{proof}

%

\begin{cor}
\label{cor:indpol}
If $f: B_1 \to B_2$ is a tropical affine immersion, then a polarization on $B_2$ induces one on $B_1$.
\end{cor}
\begin{proof}
The Riemannian metric on $B_1$ is the restriction of the one on $B_2$.
\end{proof}

\subsection{Non-existence results for tropical curves}

\begin{defn}
We say that a tropical affine torus is \emph{simple} if it admits no non-trivial sub-tropical affine torus. 
\end{defn}

\begin{lem}\label{lem:simpledense}
The set
\[ \left\{ M \in \GL(n,\R): B(M)\text{ is simple}\right\}\]
is the complement of a countable union of proper linear subspaces of the vector space $\mathrm{Mat}_{n \times n}(\R)$, intersected with the open set $\GL(n,\R) \subset \mathrm{Mat}_{n \times n}(\R)$.
\end{lem}
\begin{proof}
It is clear that $B(M)$ is non-simple if and only if there is a non-trivial rational subspace $V \subset \R^n$ such that $M\cdot V$ is also rational. (A subspace $V \subset \R^n$ is called rational if it has a basis of vectors in $\Q^n$.) 
Equivalently, for some $0<i<n$ there exist linearly independent vectors $v_1,\ldots,v_i \in \Q^n$ and linearly independent vectors $w_1,\ldots,w_{n-i} \in \Q^n$ such that $M$ lies in the proper rational subspace of $\mathrm{Mat}_{n \times n}(\R)$ defined by $w_j^T \cdot M \cdot v_k = 0$ for all $j,k$. 
There are clearly a countable set of choices of such $\{v_j\},\{w_k\}$, so the proof is complete.
\end{proof}

\begin{lem}
If $n \ge 2$, the set 
\[\left\{ M \in \GL(n,\R): B(M) \text{ admits a polarization}\right\}\]
is a countable union of open subsets of proper linear subspaces of the vector space $\mathrm{Mat}_{n \times n}(\R)$, intersected with the open subset $\GL(n,\R)$.
\end{lem}
\begin{proof}
Let
\begin{align*}
 SPD_n &:= \{g \in \mathrm{Mat}_{n \times n}(\R): \text{ $g$ symmetric positive definite}\}\\
 NZD_n &:= \{A \in \mathrm{Mat}_{n \times n}(\Z): \det(A) \neq 0\}.
\end{align*}
By Lemma \ref{lem:whenmpol} (and using the fact that a matrix is symmetric positive definite if and only if its inverse is so), $M$ admits a polarization if and only if it lies in the image of the map
\begin{equation}
\label{eqn:spdnzd}
SPD_n \times NZD_n \to \GL(n,\R): (g,A) \mapsto g\cdot A.
\end{equation}
This map is bilinear in the entries of $SPD_n$ and $NZD_n$. 
Note that $SPD_n$ is an open subset of the set of symmetric matrices, which is a linear subspace of dimension $n(n+1)/2 < n^2$; so for fixed $A$, the image of the map is an open subset of a proper linear subspace. 
We now observe that $NZD_n$ is countable, concluding the proof.
\end{proof}

\begin{cor}
If $n \ge 2$, the set
\begin{equation}
\label{eq:simpnonpol}
 \{M \in \GL(n,\R): \text{ $B(M)$ is simple and does not admit a polarization}\}
 \end{equation}
is dense in $\GL(n,\R)$.
\end{cor}

\begin{lem}
\label{lem:nocurve}
If $B$ is a tropical affine torus such that $\Pic^0(B)$ is simple and does not admit a polarization (in particular, if $B=B(M)$ where $M^T$ is an element of the set \eqref{eq:simpnonpol}), then $B$ admits no non-trivial tropical curves. 
\end{lem}
\begin{proof}
Let $h: C \to B$ be a tropical curve. 
Then we have a map
\[h^*: \Pic^0(B) \to \Jac(C).\]
Since $\Pic^0(B)$ is simple, $h^*$ must either be the zero map, or an immersion. 

In the first case, it would follow that $h^*: H^1(B;\R) \to H^1(C;\R)$ vanishes, so $C$ lifts to the universal cover of $B$. 
It is easy to check that the balancing condition forces a compact tropical curve in $\R^n$ to be constant.

In the second case, the polarization on $\Jac(C)$ would induce one on $\Pic^0(B)$ by Corollary \ref{cor:indpol}, contradicting the hypothesis.
\end{proof}

\begin{cor}\label{cor:Chow_free}
Under the same conditions, there are no non-trivial effective linear equivalences between tropical $0$-cycles: so $\CH_0(B) \simeq \Z^{B}$. 
\end{cor}

\begin{proof} An effective linear equivalence, by definition, arises from a tropical curve in $\R\times B$. The projection of any connected component of this tropical curve to $B$ is still balanced, so must be a point. The tropical curve $\R\times\{pt\} \subset \R\times B$  induces a trivial relation on $0$-cycles.  \end{proof}


We conclude with a criterion for $B(M)$ to satisfy the conditions of Lemma \ref{lem:nocurve}.
Observe that if $M\in \GL(n,\R)$ is symmetric, it has $n(n+1)/2$ independent entries; if these entries are rationally independent we say that $M$ is an \emph{irrational symmetric} matrix.

\begin{lem}
\label{lem:irratsimple}
If $M$ is irrational symmetric and invertible, then $\Pic^0(B(M))$ is simple.
\end{lem}
\begin{proof}
Recall that $\Pic^0(B(M)) \simeq B(M^T)$. 
If $B(M^T)$ were not simple, we would have $w^T \cdot M \cdot v = 0$ for some non-zero vectors $w,v \in \Q^n$ (see the proof of Lemma \ref{lem:simpledense}). 
This would imply $\sum_{ij} w_i M_{ji} v_j = 0$, which implies $w_iv_j+w_jv_i=0$ for all $i,j$ because $M$ is irrational symmetric. 
Since $v \neq 0$, let us suppose $v_i \neq 0$; then $w_iv_i=0$ so $w_i=0$; then $w_iv_j+w_jv_i=0$ implies $w_j=0$ for all $j\neq i$; and we have derived a contradiction since $w \neq 0$.
\end{proof}

\begin{lem}\label{lem:irratpol}
If $M$ is irrational symmetric, invertible and indefinite, then $\Pic^0(B(M))$ does not admit a polarization.
\end{lem}
\begin{proof}
Suppose that $\Pic^0(B(M)) \simeq B(M^T)=B(M)$ admits a polarization. 
Lemma \ref{lem:whenmpol} implies that  $M=g \cdot A$ for some positive-definite symmetric matrix $g$ and some matrix with integer entries  and non-zero determinant $A$. 
In particular, $g=M A^{-1}$ must be symmetric, so $M A^{-1} = (A^{-1})^TM$.
Since $M$ is irrational symmetric whereas the entries of $A^{-1}$ are rational, it follows that 
\begin{align*}
\mathsf{e}_{ii} \cdot A^{-1} &= (A^{-1})^T \cdot \mathsf{e}_{ii} \text{ and} \\
(\mathsf{e}_{ij} + \mathsf{e}_{ji}) \cdot A^{-1} &= (A^{-1})^T \cdot  (\mathsf{e}_{ij} + \mathsf{e}_{ji})
\end{align*}
where $\mathsf{e}_{ij}$ denote the elementary matrices. 
It is simple to conclude from the first equation that $A^{-1}$ is diagonal, and then from the second that the diagonal entries coincide; so $A$ is a multiple of the identity. 
But then since $M$ is indefinite, $g=M\cdot  A^{-1}$ is too, a contradiction.
\end{proof}

\subsection{Mixed-sign tropical curves}

Now we introduce the rather unfamiliar notion of a parametrization of a tropical $1$-cycle, by analogy with the notion of a parametrization of a tropical curve. 
Later we will use this notion to construct formal Lagrangian cobordisms, which our non-existence result shows can not be realized as tautologically unobstructed Lagrangian cobordisms.

\begin{defn}
A \emph{mixed-sign abstract tropical curve} is an abstract tropical curve $C$ equipped with an assignment $\epsilon: \mathrm{Edges}(C) \to \{\pm 1\}$ of signs to each edge. 
\end{defn}

A \emph{mixed-sign parametrized tropical curve} in a tropical affine manifold $B$ is a mixed-sign tropical curve $(C,\epsilon)$ together with a map $h: C \to B$ satisfying all of the conditions of a parametrized tropical curve, except that the balancing condition gets twisted by the signs $\epsilon$:
\[ \sum_{E_i \to v} \epsilon(E_i) \cdot dh(z_i) = 0.\]
The image of a mixed-sign parametrized tropical curve is a tropical $1$-cycle, and any tropical $1$-cycle admits a canonical parametrization.

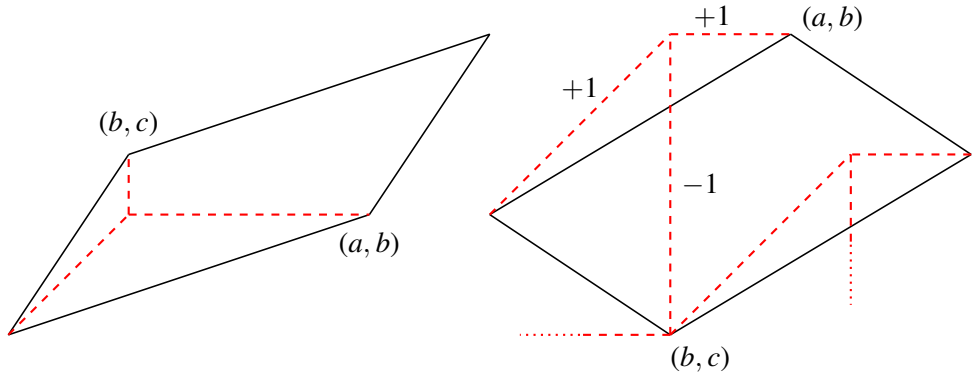
\begin{figure}[ht]
\begin{center} 
\begin{tikzpicture}[scale=0.8]

\draw[semithick] (0,0) -- (2,3);
\draw[semithick] (0,0) -- (6,2);
\draw[semithick] (2,3) -- (8,5);
\draw[semithick] (6,2) -- (8,5);
\draw (6,1.5) node {$(a,b)$};
\draw (2,3.5) node {$(b,c)$};
\draw[thick,red,dashed] (0,0) -- (2,2);
\draw[thick,red,dashed] (2,2) -- (6,2);
\draw[thick,red,dashed] (2,2) -- (2,3);

\draw[semithick] (8,2) -- (11,0);
\draw[semithick] (8,2) -- (13,5);
\draw[semithick] (13,5) -- (16,3);
\draw[semithick] (11,0) -- (16,3);
\draw (13.7, 5.25) node {$(a,b)$};
\draw (11.5,-0.4) node {$(b,c)$};
\draw[thick,red,dashed] (11,0) -- (11,5);
\draw[thick,red,dashed] (11,5) -- (13,5);
\draw[thick,red,dashed] (8,2) -- (11,5);
\draw (11.5,2.5) node {$-1$};
\draw (11.7,5.3) node {$+1$};
\draw (9.5,4.1) node {$+1$};
\draw[thick,red,dashed] (16,3) -- (14,3);
\draw[thick,red,dashed] (14,3) -- (14,1.5);
\draw[thick,red,dotted] (14,1.5)--(14,0.5);
\draw[thick,red,dashed] (11,0) -- (14,3);
\draw[thick,red,dashed] (11,0) -- (9.5,0);
\draw[thick,red,dotted] (9.5,0)--(8.5,0);

\end{tikzpicture}
\end{center}
\caption{A tropical curve (left) and mixed-sign tropical curve (right).}
\label{fig:mixed_sign}
\end{figure}

The \emph{mixed-sign integration pairing} is a pairing
\[ \int^\epsilon: H^0(T^*_\Z C) \otimes H_1(C;\Z) \to \R\]
which is the same as the usual one, except one multiplies the component of the integral over edge $E$ by the sign $\epsilon(E)$. 
This allows us to define the \emph{mixed-sign Jacobian} $\Jac(C,\epsilon)$, which is a tropical affine torus so long as the mixed-sign integration pairing is nondegenerate.
The \emph{mixed-sign Abel--Jacobi map} is a mixed-sign parametrized tropical curve $\AJ:(C,\epsilon) \to \Jac(C,\epsilon)$.

\begin{lem}
\label{lem:mscurvenopol}
There exist tropical affine tori which are simple and do not admit a polarization, but which admit mixed-sign parametrized tropical curves.
\end{lem}
\begin{proof}
We can construct examples as $\Jac(C,\epsilon)$ where $(C,\epsilon)$ is a mixed-sign tropical curve. 
These certainly admit mixed-sign parametrized tropical curves via the mixed-sign Abel--Jacobi map; so it suffices by Lemmas \ref{lem:irratsimple} and \ref{lem:irratpol} to arrange that the mixed-sign integration pairing is irrational symmetric, nondegenerate and indefinite (where we implicitly identify $H_1(B;\Z) \simeq H^0(T^*_\Z B)$ by Serre duality). 
This is easy to do in examples: e.g., we can take $C$ to have two vertices $v_1,v_2$ and three edges $E_1,E_2,E_3$ connecting $v_1$ to $v_2$, and $\epsilon(E_i)$ negative for $i=2$ only. 
If we orient these edges to all point from $v_1$ to $v_2$ then the proof of Lemma \ref{lem:jacpol} gives an identification
\[ H_1(C;\Z) \simeq H^0(T^*_\Z C) \simeq \left\{(a_1,a_2,a_3) \in \Z^3: \sum_i a_i = 0\right\},\]
and the integration pairing is identified with the restriction of the bilinear form $\mathrm{diag}(\ell_1,-\ell_2,\ell_3)$ to this subspace, where $\ell_i>0$ is the length of $E_i$.
The restricted bilinear form is irrational symmetric if and only if the $\ell_i$ are rationally independent, and it has signature $(+,-)$ if and only if $\ell_2 > (\ell_1^{-1} + \ell_3^{-1})^{-1}$; it is clearly possible to satisfy these conditions simultaneously, cf.  Figure \ref{fig:mixed_sign}.
\end{proof}

\begin{rmk}
\label{rmk:diffchow}
Lemma \ref{lem:mscurvenopol} shows the difference between the equivalence relations of linear equivalence and effective linear equivalence (compare Remark \ref{rmk:weirdchow}). 
Namely, for a tropical affine torus as in Lemma \ref{lem:mscurvenopol},  effective linear equivalence is a trivial equivalence relation on $0$-cycles, but linear equivalence is not.
\end{rmk}

\subsection{Theta-divisors}

\begin{defn}
A \emph{tropical divisor} in a tropical affine manifold $B$ is a codimension-$1$ tropical subvariety.
\end{defn}

We observe that a tropical divisor is equivalent to a choice of section $f \in H^0(C^0_B/\Aff_B)$ which looks locally like $\max_i{f_i(b)}$ for some $f_i \in \Aff_B$. 
Given such a section $f$, the corresponding tropical divisor $\Trop(f)$ is  defined to be the `bend locus' of $f$, equipped with (positive) weights corresponding to the `amount of bending' of $f$ at the codimension-1 cells. 

Now let $B$ be a tropical affine torus equipped with a polarization $c: H_1(B;\Z) \to H^0(T^*_\Z B)$. 
In the following we will write $T^*_\Z B$ for $H^0(T^*_\Z B)$, and similarly for $T^*_\R B$, $T_\Z B$, $T_\R B$. 
 
Let $k \in \N$, and let $\delta: T^*_\Z B \to \R$ be a function satisfying $ \delta(\alpha + k\cdot c(\gamma)) = \delta(\alpha)$
for all $\gamma \in H_1(B;\Z)$. 
We define a corresponding piecewise-linear function 
\begin{align*}
f_{k,\delta}: T_\R B & \to \R\\
f_{k,\delta}(v) &:= \max_{\alpha \in T^*_\Z B} \left\{\alpha(v) - \frac{\left|\alpha^\#\right|^2}{2k} + \delta(\alpha)\right\},
\end{align*}
where the length of $\alpha^\#$ is calculated with respect to the Riemannian metric induced by the polarization. 
We have a corresponding tropical divisor $\Trop(f_{k,\delta})$ in $T_\R B$. 

We now show that this divisor descends to $B \simeq T_\R B/\im(\ints)$.
To do this we observe that the function $f_{k,\delta}$ has the following `quasi-periodicity' property:

\begin{lem}
If $\gamma \in H_1(B;\Z)$, then $f_{k,\delta}(v +  \ints(\gamma)) - f_{k,\delta}(v)$ is affine.
\end{lem}
\begin{proof}
We have
\begin{align*}
f_{k,\delta}(v+\ints(\gamma)) &= \max_\alpha \left\{\alpha(v+\ints(\gamma)) - \frac{\left|\alpha^\#\right|^2}{2k}+ \delta(\alpha)\right\} \\
&= \max_\alpha \left\{\alpha(v) + \left \langle \alpha^\#, \ints(\gamma)\right\rangle - \frac{\left|\alpha^\#\right|^2}{2k}+ \delta(\alpha)\right\} \\
&= \max_\alpha \left\{\alpha(v) - \frac{\left|\alpha^\#-k \cdot\ints( \gamma)\right|^2}{2k} +  \frac{k \cdot \left|\ints(\gamma)\right|^2}{2} + \delta(\alpha)\right\} \\
&= \max_\alpha \left\{\left(\alpha-k \cdot c(\gamma)\right)(v) - \frac{\left|(\alpha-k \cdot c(\gamma))^\#\right|^2}{2k}+ \delta(\alpha)\right\} +k \cdot c(\gamma)(v) +  \frac{k \cdot \left|\ints(\gamma)\right|^2}{2}\\
&= f_{k,\delta}(v) +k \cdot c(\gamma)(v) + \frac{k\cdot \left|\ints(\gamma)\right|^2}{2},
\end{align*}
where the final line follows from the fact that the map $\alpha \mapsto \alpha - k \cdot c(\gamma)$ is a bijection on $T^*_\Z B$ which preserves the function $\delta$.
\end{proof}

Since translating by $\ints(\gamma)$ adds an affine function to $f_{k,\delta}$, it preserves the tropical divisor $\Trop(f_{k,\delta})$, which therefore descends to $B$. 
Such a tropical divisor in $B$ is called a \emph{tropical $\Theta$-divisor} \cite{Mikhalkin2008}.

Now, given a tropical divisor $\Trop(f)$ and a point $b \in B$, there is a corresponding lattice polytope $\mathrm{in}_b(f) \subset T^*_{\R} B$, which is defined to be the convex hull of the differentials $df_i$ for all $i$ such that $f_i(b)$ is maximal. 
We observe that up to translation, the polytope depends only on the tropical divisor. 
We recall that a lattice simplex is called \emph{regular} if it contains no lattice points other than its vertices. 
A tropical divisor corresponding to a section $f$ will be called \emph{regular} if $\mathrm{in}_b(f)$ is a regular simplex for all $b \in B$.
 
\begin{lem}
\label{lem:regdiv}
Let $B$ be a polarized tropical affine torus. 
Then it admits a regular tropical divisor with at least one 0-cell.
\end{lem}
\begin{proof}
We will show that there exist $k,\delta$ such that the tropical $\Theta$-divisor $\Trop(f_{k,\delta})$ is regular and has at least one 0-cell. 
Indeed, it is a simple matter to check that \emph{any} tropical $\Theta$-divisor contains a $0$-cell (a cell of minimal dimension must be a sub-tropical affine torus along which $f_{k,\delta}$ is linear; however $f_{k,\delta}$ has quadratic growth), so the real task is to find $k$ and $\delta$ such that $\Trop(f_{k,\delta}$) is regular.

It will be helpful to consider the Legendre transform of $f_{k,\delta}$ \cite[Proposition 2.1]{Mikhalkin2004}, which is the largest convex function $g_{k,\delta}:T^*_\R B \to \R$ for which 
\[g_{k,\delta}(\alpha) \le \frac{\left|\alpha^\#\right|^2}{2k} - \delta(\alpha)\]
for all $\alpha \in T^*_\Z B$. 
This function is piecewise-linear, and its bend locus defines a polyhedral decomposition of $T^*_\R B$ whose facets are precisely the sets $\mathrm{in}_b(f_{k,\delta})$. 
We need to show that $k$ and $\delta$ can be chosen so that the cells of this polyhedral decomposition are regular simplices.
 
Let us start by setting $k=1, \delta = 0$. 
We observe that, by the strict convexity of the function $|\alpha^\#|^2$, each lattice point in $T^*_\Z B \subset T^*_\R B$ is then a 0-cell in the polyhedral decomposition. 
We next observe that any sufficiently small perturbation of $\delta$ results in a subdivision of the polyhedral decomposition. 
In fact, for any cell in the polyhedral decomposition, a generic small perturbation of the values of $\delta$ on its vertices would result in a subdivision of the cell into simplices (compare \cite{OP}). 
Since our starting decomposition included all lattice points as $0$-cells, the same is true of the subdivision; it follows that each simplex in the subdivision would in fact be regular.

However, it may not be possible to perturb the values of $\delta$ at these points independently while respecting the periodicity condition on $\delta$.
To resolve this, we must increase $k$. 
We observe that the polyhedral decomposition of $T^*_\R B$ induced by $g_{k,0}$ is independent of $k$, and invariant under translation by $c(\gamma)$ for $\gamma \in H_1(B;\Z)$. 
It is an elementary matter to check that each cell is compact, and there are finitely many up to translation; so by increasing $k$ if necessary we can ensure that the star of each vertex projects injectively to $T^*_\R B/\im(k \cdot c)$. 
Once this is the case, we can independently perturb the coefficients $\delta$ attached to the vertices of each cell in our polyhedral decomposition, and a generic small perturbation will result in a subdivision into regular simplices as required.
\end{proof}

\subsection{Existence results for tropical curves}

The aim of this section is to prove the following tropical analogue of \cite[Lemma 1.3]{Bloch1976}:

\begin{prop}
\label{prop:chowdiv}
Let $B$ be a polarized tropical affine torus. 
Then $\CH_0(B)_\hom$ is divisible.
\end{prop} 

In fact, for technical reasons we will need a slight strengthening of this result. 
We define two $0$-dimensional tropical subvarieties to be \emph{trivalently linearly equivalent} if they are linearly equivalent, and the tropical curve $V$ realizing the linear equivalence has only trivalent vertices. 
We define $\CH^\tri_0(B)$ to be the quotient of the free abelian group generated by $0$-dimensional tropical subvarieties by effective trivalent linear equivalence. 
Thus there is a surjective map $\CH_0^\tri(B) \twoheadrightarrow \CH_0(B)$ which need not be an isomorphism. 

\begin{prop}
\label{prop:chowdivtri}
Let $B$ be a polarized tropical affine torus. 
Then $\CH_0^\tri(B)_\hom$ is divisible.
\end{prop}

The proof will use two Lemmas:

\begin{lem}
\label{lem:somecurvesexist}
Let $B$ be a polarized tropical affine torus, and $b_0 \in B$. 
Then there is a non-empty open set $V \subset B$ such that, for all $b \in V$, the points $b_0$ and $b$ lie on edges of a trivalent tropical curve $C$ in $B$. 
\end{lem}
\begin{proof}
Recall that $n$ denotes the dimension of $B$.
The case $n=1$ is trivial; next we consider the case $n=2$, which contains the kernel of the general idea.
There exists a regular tropical divisor $D \subset B$ with at least one $0$-cell, by Lemma \ref{lem:regdiv}. 
In this dimension a regular tropical divisor is equivalent to a trivalent tropical curve. 
Let us translate $D$ so that $b_0$ lies on an edge of $D$, close to a $0$-cell $p$. 
Let us place $b$ on another edge of $D$, also close to $p$: then we can move $b$ along its edge, and we can translate the curve $D$ parallel to the edge on which $b_0$ lies, so that $b$ sweeps out an open set $V$.

Now let us consider the case $n \ge 3$; as before we have a regular tropical divisor $D \subset B$ with a $0$-cell $p$, by Lemma \ref{lem:regdiv}. 
Given $v=(v_1,\ldots,v_{n-1}) \in (T_\R B)^{n-1}$, we consider the intersection
\[ C(v) := \bigcap_{i=1}^{n-1} (D+v_i).\]
We say that the intersection is \emph{transverse} if each intersection of cells is transverse; this holds if $v$ is sufficiently small and lies in the complement of a finite union of proper linear subspaces.
When the intersection is transverse, $C(v)$ comes naturally equipped with the structure of a tropical curve \cite[Section 4.2]{Mikhalkin2006}. 
The transverse intersection of \emph{regular} tropical divisors is furthermore a \emph{trivalent} tropical curve. 

Now in a sufficiently small neighbourhood $N$ of the $0$-cell $p$, $D$ looks like $\Trop(f)$ where $f = \max_{i=1,\ldots,n+1} \{f_i\}$ with $df_i \in T^*_\Z B$ forming a top-dimensional simplex. 
Inside $N$ the cells of $D$ are indexed by the subsets $K \subset \{1,\ldots, n+1\}$ such that $|K| \ge 2$, with the cell $D_K$ corresponding to the locus where $\{f_i = f \, \forall \, i \in K\}$ (and having codimension $|K|-1$). 
If we choose sufficiently small generic $v$ so that $C(v)$ is a transverse intersection, then $N \cap C(v)$ will be a tropical curve with non-compact edges $E_i(v)$ exiting $N$ parallel to the one-cells $D_{\overline{\{i\}}}$, for $i=1,\ldots,n+1$. 

By an appropriate choice of $v$ we may ensure that $E_1(v)$ is locally the transverse intersection of $D_{\{i,n+1\}}+v_i$ for $i = 2,\ldots,n$. 
We translate so that $b_0$ lies on $E_1(v)$ and choose $b$ lying on $E_{n+1}(v)$. 
We will now show that it is possible to vary $v$ so that $E_1(v)$ continues to pass through $b_0$ but $b$ sweeps out an open set.

Indeed, if we choose $v+\epsilon = (v_1+\epsilon_1,\ldots,v_{n-1}+\epsilon_{n-1})$ so that $\epsilon_i \in T D_{\{i,n+1\}}$, then $E_1(v) = E_1(v+\epsilon)$ is unchanged near $b_0$ for $\epsilon$ sufficiently small. 
In particular $b_0$ still lies on $C(v+\epsilon)$. 
At the point $b \in E_{n+1}(v)$ we have
\[NE_{n+1}(v) \simeq \bigoplus_{i=1}^{n-1} ND_{K_i}\]
by transversality, where $K_1,\ldots,K_{n-1} \subset \{1,\ldots,n+1\}$ are distinct two-element subsets not containing $n+1$ (which subsets $K_i$ arise depends on the choice of $v$). 
Note that $\{i,n+1\} \neq K_i$ (since $n+1 \notin K_i$), so the projection $TD_{\{i,n+1\}} \to ND_{K_i}$ is surjective for each $i$. 
So by varying $\epsilon_i \in TD_{\{i,n+1\}}$ we can move the edge $E_{n+1}(v)$ in any normal direction we desire; combined with the fact we can move $b$ along the edge $E_{n+1}(v)$, it follows that $b$ sweeps out an open set $V$ as required.
\end{proof}

\begin{lem}
\label{lem:curvesexist}
Let $B$ be a polarized tropical affine torus, and $b_0 \in B$. 
Then there is an open dense set $U \subset B$ such that, for all $b \in U$, the points $b_0$ and $b$ lie on edges of a trivalent tropical curve $C$ in $B$. 
\end{lem}
\begin{proof}
We have shown in Lemma \ref{lem:somecurvesexist} that there exists an open set $V \subset B$ with the desired property. 
We observe that for any $k \in \N$ we have an isogeny $\pi_k: B \to B$ given by multiplication by $k$ (and sending $b_0 \mapsto b_0$). 
If $C$ is a trivalent tropical curve with $b_0$ and $b$ on its edges, then $\pi_k^{-1}(C)$ is a trivalent tropical curve with $b_0$ and $b'$ on its edges, for any $b' \in \pi_k^{-1}(b)$. 
It follows that the set
\begin{equation}
\label{eqn:denseU}
 U:= \bigcup_{k \in \N} \pi_k^{-1}(V)
\end{equation}
has the desired property. 
This set is obviously open, and it is furthermore dense. 
In fact the set \eqref{eqn:denseU} is dense whenever $V$ is non-empty. 
To see this, choose a flat Riemannian metric on $B$ so that the image under the exponential map of the ball of radius $1$ in $T_\R B$ contains a fundamental domain of $B$.
Then given a point $p \in V$, any ball of radius $1/k$ in $B$ contains an element of $\pi_k^{-1}(\{p\})$. 
It follows that the union over $k$ of these sets is dense. 
\end{proof}

\begin{proof}[Proof of Proposition \ref{prop:chowdivtri}]
It suffices to prove that, for any $b_1, b_2 \in B$, the class $[b_2 - b_1]$ is infinitely divisible in $\CH_0^\tri(B)_\hom$. 
This will be achieved if we show that there exists a point $b_3$ such that both $[b_3-b_1]$ and $[b_2-b_3]$ are infinitely divisible in the same group. 
In particular it suffices to prove this for any $b_3$ in $(U + b_1 - b_0) \cap (U + b_2-b_0)$, where $b_0 \in B$ and $U$ is the set constructed in Lemma \ref{lem:curvesexist}, since the latter is open and dense and in particular the intersection is non-empty.

Thus, it suffices to prove that the class $[b-b_0]$ is infinitely divisible for any $b \in U$. 
By construction, there exists a trivalent parametrized tropical curve $h:C \to B$ with points $p, p_0 \in C$ lying on its edges, such that $h(p) = b$ and $h(p_0) = b_0$. 
Since $\CH_0(C)_\hom$ is isomorphic to the tropical affine torus $\Jac(C)$ via the Abel--Jacobi map, the class $[p-p_0]$ is infinitely divisible. 
Thus for any $k \in \N$ there exists a tropical divisor $a = \sum_q a_q \cdot q$ with $k\cdot a$ linearly equivalent to $p-p_0$. 
It follows that $[b-b_0] = h_*[p-p_0] = k \cdot h_* (a)$ is $k$-divisible in $\CH_0(B)_\hom$. 

However we want to check that $[b-b_0]$ is $k$-divisible in $\CH_0^\tri(B)_\hom$. 
Although $C$ is trivalent, the graph construction of Section \ref{subsec:paramcurve} adds certain edges to $C$ so that the resulting tropical curve may no longer be trivalent. 
However, by examining the construction, we see that to preserve trivalency it suffices for the points at which edges are attached to $C$ to be interior to edges of $C$. 
The points $p$ and $p_0$ are interior to the edges of $C$ by construction; and it is easy to check that any tropical divisor $a$ on a curve is linearly equivalent to one whose points are interior to the edges. 
Making this replacement, we have shown that $[b-b_0]$ is $k$-divisible in $\CH_0^\tri(B)_\hom$ for any $k$, thus completing the proof.
\end{proof}

\section{From curves to cobordisms} 
\label{Sec:CurvetoCob}

\subsection{Lagrangian lifts}

Let $B$ be a tropical affine manifold and $X(B) := T^* B/T^*_\Z B$ the corresponding symplectic manifold, with $p: X(B) \to B$ the Lagrangian torus fibration. 
Let $C \subset B$ be a tropical $1$-cycle.
A \emph{choice of vertex neighbourhoods} for $C$ will consist of a disjoint union $U = \sqcup_v U_v$ of balls centred at the vertices $v$ of $C$, which are small enough that each edge of $C$ intersects $U$ only in a neighbourhood of its endpoints. 

\begin{defn}
An \emph{immersed Lagrangian lift} of a tropical $1$-cycle $C \subset B$ is a Lagrangian immersion $i: L \to  X(B)$, proper over $B$,  such that there exists a choice of vertex neighbourhoods $U$ with the property that, outside of $p^{-1}(U)$, $L$ coincides with the union over the edges $E$ of $w_E$ disjoint parallel copies of the conormal to $E$. 
An \emph{embedded Lagrangian lift} is an immersed Lagrangian lift such that $i$ is an embedding.
\end{defn}

\begin{rmk}
Different notions of Lagrangian liftability are introduced in \cite{Matessi2018,Mikhalkin2018}, where restrictions are placed on the form of the Lagrangian $L$ near a vertex; our definition (which is rather ad-hoc) does not place such restrictions.
\end{rmk}

We observe that the question of constructing Lagrangian lifts is completely local:

\begin{lem}
\label{lem:gluelags}
A tropical $1$-cycle $C$ admits an immersed (respectively, embedded) Lagrangian lift if and only if a neighbourhood of each of its vertices admits an immersed (respectively, embedded) Lagrangian lift.  
\end{lem}
\begin{proof}
The `only if' direction is obvious. 
For the `if' direction, the only concern would be if the $w_E$ parallel copies of the conormal to $E$ did not `match up' at the two ends of $E$. 
However this can easily be fixed by inserting a `twist region' along $E$ where the parallel copies $L_1,\ldots,L_{w_E}$ of the conormal to $E$ are modified by applying the flow of Hamiltonian vector fields $X_{f_i}$ associated to appropriate functions $f_i$ pulled back from the base $B$ (and then enlarging the choice of vertex neighbourhoods to engulf the twist regions).  
\end{proof}

\subsection{Local exactness}

Given an immersed Lagrangian lift $L$ of a tropical $1$-cycle $C$ with an appropriate choice of vertex neighbourhoods $U_v$, we have a covering of $L$ by open sets $L_v := (p\circ i)^{-1} (U_v)$ associated to the vertices $v$ of $C$, together with open sets $L_E$ associated to the edges $E$ of $C$, over which $L_E$ is a disjoint union of $w_E$ parallel copies of the conormal to $E$. 
We denote by $i_v: L_v \to X(B)$, $i_E: L_E \to X(B)$ the corresponding Lagrangian immersions.

On the preimage $p^{-1}(U_b)$ of any simply-connected neighbourhood $U_b$ of a point $b \in B$, there exists a primitive $\lambda_b := \sum_i (q_i - b_i)\cdot dp_i$ for the symplectic form $\omega$. 
If $i: L \to X(B)$ is an immersed Lagrangian lift, then we observe that for each vertex $v$ of $C$, $i_v^* \lambda_v$ is a closed one-form.

\begin{defn}\label{defn:locally_exact}
An immersed Lagrangian lift is called \emph{locally exact} if the classes $[i_v^*\lambda_v] \in H^1(L_v;\R)$ vanish for all vertices $v$.
\end{defn}

\begin{rmk}
We observe that in fact $i_v^*\lambda_v$ vanishes near the boundary. 
Thus we can declare $L$ to be \emph{locally strongly exact} if $[i_v^*\lambda_v] = 0$ in $H^1(L_v,\partial L_v;\R)$. 
Note however that inserting `twist regions' as in the proof of Lemma \ref{lem:gluelags} will in general destroy local strong exactness: for this purpose it matters where one puts the parallel copies of the conormal. 
\end{rmk}


\subsection{Gradings\label{Sec:gradings}}

Let $B$ be an oriented tropical affine manifold. 
Observe that we have a canonical decomposition $T(X(B)) \simeq TB \oplus T^* B$. 
An almost-complex structure is said to be \emph{split} if it takes $TB$ to $T^*B$ and vice-versa. 
We will consider gradings with respect to the holomorphic volume form $\eta_J$ corresponding to an $\omega$-compatible split almost-complex structure $J$ (cf. Section \ref{Sec:AlbCob}).

\begin{lem}
\label{lem:speclagsplit}
Let $C \subset B$ be a tropical affine submanifold of dimension $d$, and $N(C) \subset X(B)$ the conormal. 
Then the orientation on $B$ induces one on $N(C) \subset X(B)$, and $N(C)$ is special Lagrangian of phase $d\pi/2$ with respect to $\eta_J$, for any split almost-complex structure $J$.
\end{lem}
\begin{proof}
We may choose oriented local affine coordinates $b_1,\ldots,b_n$ with respect to which $C=\{b_1 = \ldots = b_{n-d} = 0\}$. 
Then we have local Arnold--Liouville coordinates $(b_1,\theta_1,\ldots,b_n,\theta_n)$ on $X(B)$, and $N(C)= \{b_1 = \ldots = b_{n-d} = \theta_{n-d+1} = \ldots = \theta_n = 0\}$.
We declare the local coordinates $(c_1,\ldots,c_n) = (\theta_1,\ldots,\theta_{n-d},b_{n-d+1},\ldots,b_n)$ on $N(C)$ to be oriented. 

We now observe that $d^cb_j = \sum_k a_{jk} d\theta_k$ because $J$ is split, and $a_{jk}$ is symmetric positive definite because $J$ is $\omega$-compatible.
Therefore the restriction of $\eta_J$ to $N(C)$ is
\begin{align*} 
\eta_J &= \left(\sum_k a_{1k} d\theta_k + idb_1\right) \wedge \ldots \wedge \left(\sum_k a_{nk} d\theta_k + idb_n\right) \\
&= \left(\sum_{k=1}^{n-d} a_{1k}d\theta_k\right) \wedge \ldots \wedge \left(\sum_{k=1}^{n-d} a_{n-d,k}d\theta_k\right) \wedge idb_{n-d+1} \wedge \ldots \wedge idb_n \\
&= i^d \cdot \det\left(a_{ij}\right)_{1 \le i,j \le n-d} \cdot d\theta_1 \wedge \ldots d\theta_{n-d} \wedge db_{n-d+1} \wedge \ldots \wedge db_n.
\end{align*}
Since $a_{ij}$ is positive definite, its upper left $(n-d) \times (n-d)$ minor is positive definite and in particular has positive determinant. 
Therefore the phase of $N(C)$ is constant equal to the argument of $i^d$, which is $d\pi/2$ as required.
\end{proof}

Now let $L$ be a Lagrangian lift of a tropical $1$-cycle $C$ in an oriented tropical affine manifold $B$. 
We equip the components $L_E$ lying outside the vertex neighbourhoods with orientations: they coincide with the one determined by the orientation of $B$ if the weight of $E$ is positive, and are opposite to it when the weight is negative. 
Then the components $L_E$ are graded by Lemma \ref{lem:speclagsplit}: we may choose the almost-complex structure $J$ to be split, and the phase function to be constant equal to $\pi/2$ over the edges of positive weight, and constant equal to $-\pi/2$ over the edges of negative weight.

\begin{defn}
\label{lem:locgrad}
Suppose $B$ is oriented. 
A Lagrangian lift $L$ is \emph{locally graded} if each Lagrangian immersion $i_v: L_v \to X(B)$ admits an orientation and a grading $\theta^\#$, which agree near the boundary with the orientations and gradings $\pm \pi/2$ over the regions $L_E$ described above.
\end{defn}

It is clear from our discussion that a locally graded Lagrangian lift is graded (we observe that the phase function will take the same constant value on either side of a `twist region', even though it may vary in between).

\subsection{Lagrangian lifts of tropical curves}

The following result has been known to experts for some time. 
The first-named author heard it in a talk by Denis Auroux at MSRI in 2009, and related results appear in \cite{Shende2015,Matessi2018,Mikhalkin2018,Mak-Ruddat,Hicks2018} (it is by no means the main result proved in any of these references.)

\begin{thm}
\label{thm:trivconst}
Let $C$ be a tropical curve in $\R^n$ with a single trivalent vertex. 
Then $C$ admits an embedded, locally strongly exact, locally graded Lagrangian lift.
\end{thm}
\begin{proof}
An embedded Lagrangian lift of the standard tropical pair of paints in $\R^2$ is constructed in \cite{Mikhalkin2018}, which is locally graded because it is a small perturbation of a special Lagrangian submanifold. 
One easily verifies that it is locally strongly exact, as a basis for $H_1(L,\partial L)$ is given by the `seams' of the pants, which are its intersection with a component of the real locus, along which $\lambda_v$ vanishes. 
The generalization to an arbitrary trivalent vertex in $\R^2$ is made in \cite[Example 5.3]{Matessi2018} by taking a cover of the standard pair of pants: it is straightforward that the properties of being embedded, locally strongly exact, and locally graded transfer to covers. 
The generalization to arbitrary $n$ follows by taking a product with the Lagrangian torus fibre over the origin in $T^*(\R^{n-2})/T^*_\Z(\R^{n-2})$. 
\end{proof}
%

\begin{cor}
\label{cor:trivlift}
Any trivalent tropical curve admits an embedded, locally exact, graded Lagrangian lift. 
\end{cor}
\begin{proof}
The local lifts of Theorem \ref{thm:trivconst} can be glued together by Lemma \ref{lem:gluelags}.
\end{proof}

\begin{thm}
\label{thm:cyclesimmlift}
Any tropical $1$-cycle admits an immersed, locally exact, graded Lagrangian lift.
\end{thm}
\begin{proof}
As before it suffices by Lemma \ref{lem:gluelags} to consider the local case, so we assume that our tropical $1$-cycle has a single vertex. 
The proof for a single vertex is by strong induction on the valency of the vertex. 
The base case is when the valency is three. 
If all three edges have the same sign, then the lift exists by Theorem \ref{thm:trivconst}. 
If the edges have different signs, we must do something different.

Suppose, without loss of generality, that two edges are positive and one negative.   
Let $h: (C,\epsilon) \to B$ be a mixed-sign parametrization of a neighbourhood of the vertex. 
We modify $(C,\epsilon)$ by introducing an additional bivalent vertex a finite distance along the negative edge, and declaring the newly-created finite edge to have positive weight. 
Call the new mixed-sign abstract tropical curve $(C',\epsilon')$. 
There is a corresponding mixed-sign parametrized tropical curve $h':(C',\epsilon') \to B$ which coincides with $h$ outside a compact region (see Figure \ref{Fig:Immersed_lift}). 
It has one trivalent vertex with all three edges positive, and one bivalent vertex with one positive and one negative edge. 
The trivalent vertex admits a locally strongly exact locally graded Lagrangian lift by Theorem \ref{thm:trivconst}, and the bivalent edge admits a locally strongly exact locally graded Lagrangian lift as illustrated in Figure \ref{Fig:Immersed_lift}. 
The two pieces combine (as in Lemma \ref{lem:gluelags}) to give an immersed Lagrangian lift of a neighbourhood of the vertex, and one easily sees that it is locally strongly exact and locally graded.
This completes the base case of the induction.

\begin{figure}
\begin{center}
\includegraphics[width=0.25\textwidth]{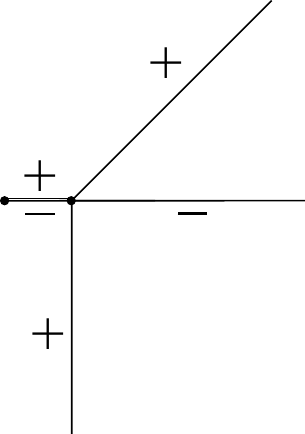}
\hspace{2cm}\raisebox{1cm}{\includegraphics[width=0.3\textwidth]{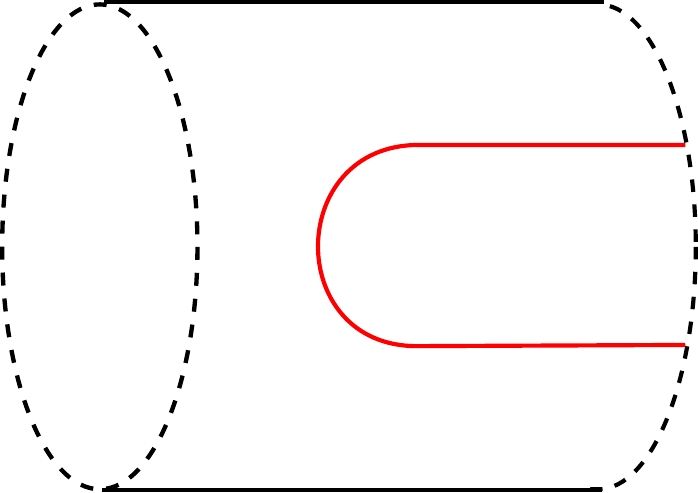}}
\end{center}
\caption{The picture on the left shows how to add a bivalent vertex to remove a mixed-sign trivalent vertex. The picture on the right shows a Lagrangian lift of a bivalent vertex. Note that the curve must bend in the right direction to be locally graded, and it must enclose the right amount of area to be locally strongly exact.}
\label{Fig:Immersed_lift}
\end{figure}

Now suppose that the valency is $\ge 4$, and $h: (C,\epsilon) \to B$ is a mixed-sign parametrization of a neighbourhood of the vertex. 
We modify $(C,\epsilon)$ by introducing an additional compact edge separating two vertices of valency $\ge 3$, then adding a bivalent vertex in the centre of that compact edge, assigning the same finite length and opposite signs to the two resulting halves. 
Call the new mixed-sign abstract tropical curve $(C',\epsilon')$, and let $h': (C',\epsilon') \to B$ be the corresponding mixed-sign parametrized tropical curve as before (the balancing condition determines the direction of the new compact edges uniquely). 
By strong induction, the two vertices of valency $\ge 3$ admit immersed locally strongly exact locally graded Lagrangian lifts; and the bivalent vertex also admits a locally strongly exact locally graded Lagrangian lift as before. 
The three pieces combine as before to give a locally strongly exact (in particular, locally exact), locally graded immersed Lagrangian lift of a neighbourhood of the vertex.   
\end{proof}
%

\subsection{Embedded formal Lagrangian lifts of arbitrary tropical $1$-cycles}

We recall the notion of a \emph{formal Lagrangian} in a symplectic manifold, which is `the closest thing to a Lagrangian that we can study using smooth topology':

\begin{defn}
Given a symplectic manifold $(M,\omega)$, a \emph{formal Lagrangian immersion} is an immersion of a half-dimensional submanifold $i: L \to M$ such that $i^*[\omega] = 0$ in $H^2(L;\R)$, together with a homotopy-through-half-dimensional-subspaces from $TL \subset i^* TM$ to a family of Lagrangian subspaces of $i^* TM$. 
\end{defn}

\begin{thm}
\label{thm:formlift}
Any tropical $1$-cycle in a tropical affine manifold of dimension $\ge 3$ admits an embedded formal Lagrangian lift (which is furthermore graded).
\end{thm}
\begin{proof}
Given a tropical $1$-cycle, the proof of Theorem \ref{thm:cyclesimmlift} provides a graded immersed Lagrangian lift $i:L \to X(B)$. We show that $i$ can be perturbed to an embedding: since the property of being a graded formal Lagrangian persists under perturbations, the result follows.

Recall that the construction of the immersed Lagrangian lift proceeds by constructing an immersed mixed-sign parametrized tropical curve $h':(C',\epsilon') \to B$ which had only trivalent vertices all of whose edges had the same sign, and bivalent vertices with edges of opposite sign. 
Each of these vertex types admits a local embedded Lagrangian lift, so the fact that the global lift is immersed is due entirely to the fact that $h'$ is immersed. 
The crucial additional observation now is that, because $C'$ is one-dimensional while the ambient space has dimension $\ge 3$, the map $h'$ can be perturbed so that it becomes an embedding (except in a neighbourhood of the bivalent vertices), at the cost that it is no longer affine along the interiors of the edges. 
We then construct the desired perturbation of $i$ by keeping the fibre components fixed and moving the base in accordance with the perturbation of $h'$.
\end{proof}


\subsection{Consequences for cobordisms}

Let $B$ be a tropical affine manifold. 
There is an obvious map from the free abelian group $\Z^B$ on points of $B$ to the free abelian group on Lagrangian torus fibres of $X(B) \to B$, which induces an identification between the group of degree zero tropical $0$-cycles and the subgroup of homologically trivial configurations of Lagrangian torus fibres.  Let $\Cob_{\fib}(X(B)/\C^*)$ denote the group which is generated by the Lagrangian torus fibres $F_b \subset X(B)$, equipped with their natural brane structures, modulo relations induced by graded Lagrangian brane cobordisms all of whose ends are fibres.  There is a natural map
\[
\Cob_{\fib}(X(B)/\C^*) \to  \langle [F_b] \, | \, b \in B \rangle\, \subset \, \Cob(X(B)/\C^*)
\]
onto the subgroup of the cylindrical cobordism group which is generated by the fibres;  this map is onto, but not necessarily an isomorphism (since there may be cobordism relations amongst fibres which necessarily involve non-fibres). There is a map
$deg: \Cob_{\fib}(X(B)/\C^*) \to \Z$ recording the total homology class of a configuration of fibres, whose kernel we denote by $\Coh_{\fib}(X(B)/\C^*)_{\hom}$. Corollary \ref{cor:trivlift} shows that the association $b \mapsto F_b$ induces a (surjective) map
\begin{equation}
\label{eq:tropchowcob}
\CH^{\tri}_0(B)_{\hom} \longrightarrow \Cob_{\fib}(X(B)/\C^*)_{\hom}
\end{equation}
since any effective tropical linear equivalence which is realised by a trivalent tropical curve in $\R\times B$  lifts to an embedded graded Lagrangian submanifold of $X(\R\times B) = \C^*\times X(B)$ which, by construction, yields a cobordism between the corresponding fibres.

\begin{cor}[= Theorem \ref{thm:divisible}]
\label{cor:divisible}
Suppose $B$ is a tropical affine torus which admits a polarization. Then the group $\Cob_{\fib}(X(B)/\C^*)_{\hom}$ is divisible.
\end{cor}

\begin{proof} This is an immediate consequence of Proposition \ref{prop:chowdivtri} and the existence of the map \eqref{eq:tropchowcob}.
\end{proof}

\section{From cobordisms to curves}

\subsection{Rigid analytic spaces}

We briefly recall the notion of a rigid analytic space in the sense of Tate, following \cite{BGR,Einsiedler2006}.

Let $\Bbbk$ be an algebraically closed field of characteristic $2$. 
Consider the (universal one-variable) Novikov field over $\Bbbk$:
\begin{equation} \Lambda := \left\{ \sum_{j=0}^\infty c_j \cdot q^{\lambda_j}: c_j \in \Bbbk, \lambda_j \in \R, \lim_{j \to \infty} \lambda_j = +\infty\right\},\end{equation}
which is an algebraically closed field extension of $\Bbbk$. 
It has a non-Archimedean valuation
\begin{align} 
val: \Lambda & \to \R \cup \{\infty\} \\
\label{eqn:valdef} val\left(\sum_{j=0}^\infty c_j \cdot q^{\lambda_j}\right) &= \min_j\{\lambda_j: c_j \neq 0\}, \quad val(0) = \infty,
\end{align}
which gives rise to a norm $|\cdot|: \Lambda \to \R$ via $|x| = e^{-val(x)}$. The unitary subgroup 
\begin{equation} \label{eqn:unitary_novikov}
 U_{\Lambda} = val^{-1}(0) = \left \{ a + \sum_j c_j \cdot q^{\lambda_j} \, \in \Lambda \, : \, a \in \Bbbk^*, \lambda_j > 0 \right\}
\end{equation}
comprises the elements of valuation zero. 

We define the Gauss norm on the polynomial algebra $\Lambda[z_1,\ldots,z_n]$: 
\begin{equation}
\left| \sum_a f_a \cdot z^a \right| := \max_a \left\{|f_a|\right\},
\end{equation}
where $a \in \Z_{\ge 0}^n$ denotes a multi-index and $f_a \in \Lambda$. 
The completion of the polynomial algebra with respect to the Gauss norm is the Tate algebra
\[
\Lambda\langle z_1,\ldots,z_n\rangle = \left\{ \sum_a \, f_a \cdot z^a \ : \ val(f_a) \to \infty \ \textrm{as} \ \|a\| \to \infty \right\}
\]
where $\|\cdot \|$ is any norm on $\Z^n$. It is a $\Lambda$-Banach algebra, i.e. a complete non-Archimedean normed $\Lambda$-algebra, and it is Noetherian \cite[5.2.6]{BGR}. 
It can be regarded as a function algebra on the rigid analytic unit polydisc
\[B^n(\Lambda) = \{(z_1,\ldots,z_n) \in \Lambda^n \, : \, |z_i| \leq 1 \, \forall \, i\},\]
and in fact its maximal spectrum can be identified with $B^n(\Lambda)$ \cite[7.1.1]{BGR}.

An affinoid algebra $A$ is a $\Lambda$-Banach algebra which is isomorphic to a quotient of some Tate algebra. 
The maximal spectrum of an affinoid algebra comes equipped with a Grothendieck topology and a sheaf of rings $\mathcal{O}$ whose global sections recover $A$ by the Tate acyclicity theorem. 
A rigid analytic space over $\Lambda$ consists of a set $Y$ equipped with a Grothendieck topology and a sheaf of rings $\mathcal{O}_Y$, which is locally isomorphic to the maximal spectrum of an affinoid algebra. 
 
\subsection{The non-Archimedean SYZ fibration}
\label{sec:nonArchSYZ}

Now let $B$ be a tropical affine manifold. 
The mirror to the Lagrangian torus fibration $p:X(B) \to B$ is a rigid analytic space $Y(B)$, equipped with a map (of sets) $\trop: Y(B) \to B$, which is called the non-Archimedean SYZ fibration.\footnote{The terminology is from \cite{Nicaise2018}. We remark that all of our non-Archimedean SYZ fibrations have non-singular fibres, so are in fact affinoid torus fibrations in the language of \emph{op. cit.}} 
We describe the construction following \cite{Kontsevich2001,Kontsevich2006,Abouzaid:ICM,Abouzaid:FFFF}.

As a set, 
\[Y(B) := \bigsqcup_{b\in B} H^1(F_b; U_{\Lambda}) \]
is the set of $U_\Lambda$-local systems on fibres $F_b$ of $p: X(B) \to B$. 
It can be identified with the total space of the local system $T_\Z B \otimes_\Z U_\Lambda$ over $B$. 
It comes equipped with a natural map $\trop: Y(B) \to B$.
In order to equip $Y(B)$ with the structure of a rigid analytic space we must cover it with maximal spectra of affinoid algebras in an appropriate way. 

The construction is locally modelled on the case $B = \R^n$, where $Y(B) = \mathbb{G}_m^n$ and the map $\trop: Y(B) \to B$ is defined by
\[\trop(z_1,\ldots,z_n) = (val(z_1),\ldots,val(z_n)). \]
Note that the total space of $T_\Z B \otimes_\Z U_\Lambda \simeq \R^n \times U_\Lambda^n$ is in bijection with $Y(B)$ via the map
\[((\lambda_1,\ldots,\lambda_n),(u_1,\ldots,u_n))  \mapsto \left(q^{\lambda_1}\cdot u_1,\ldots,q^{\lambda_n}\cdot u_n \right).\]
We also note that this construction is functorial with respect to affine transformations of $\R^n$, in particular for any affine transformation $\varphi(x) = Ax +b$ there is a commutative diagram
\begin{equation}
\xymatrix{ \mathbb{G}_m^n \ar[r] \ar[d]^-{\trop} & \mathbb{G}_m^n \ar[d]^-{\trop} \\
\R^n \ar[r]^-{\varphi} & \R^n}
\end{equation}
with the top arrow sending 
\[(z_1,\ldots,z_n) \mapsto \left(q^{b_1}\cdot z^{A_1},\ldots,q^{b_n} \cdot z^{A_n} \right),\]
where $A_i$ denotes the $i$th column of the matrix $A \in \GL(n,\Z)$.

Now if $P \subset \R^n$ is a bounded rational polytope then we define a norm on the algebra of Laurent polynomials $\Lambda[z_1^\pm,\ldots,z_n^\pm]$ corresponding to the valuation
\begin{equation}
v_P\left( \sum_a f_a \cdot z^a \right) := \inf \left\{val(f_a) + a \cdot p: \, a \in \Z^n, p \in P\right\}.
\end{equation}
The completion of the algebra of Laurent polynomials with respect to this norm is the polytope algebra
\[ \mathcal{O}_P := \left\{ \sum_a\, f_a \cdot z^a \ : \ val(f_a) + a \cdot p \to \infty \ \textrm{as} \ \|a\| \to \infty \text{ for all $p \in P$} \right\},\]
where now $a \in \Z^n$. 
This is an affinoid algebra \cite[Proposition 3.1.5]{Einsiedler2006} and in particular Noetherian, and it can be regarded as a function algebra on the polytope domain $Y_P := \trop^{-1}(P)$. 
In fact the maximal spectrum of $\mathcal{O}_P$ can be identified with $Y_P$ \cite[Proposition 3.1.8]{Einsiedler2006}. 

It follows that for any polytope $P \subset B$, the subset $\trop^{-1}(P) \subset Y(B)$ is naturally in bijection with the maximal spectrum of an affinoid algebra $\mathcal{O}_P$. 
Taking a cover of $B$ by polytopes we obtain an admissible affinoid cover, which completes the construction of the rigid analytic space $Y(B)$. 

 \subsection{Tropicalization} An analytic subset
 of a rigid analytic space $Y$ is, by definition, the zero-locus of a coherent ideal sheaf, cf. \cite[9.5.2]{BGR} or \cite{Conrad}.  These are precisely the reduced subspaces $Z \subset Y$ which are locally, i.e. on the sets of an admissible affinoid cover, given as the zero-set of a finite collection of convergent power series.  Such an analytic subset has a well-defined dimension (at a point this is the Krull dimension of the local ring $\mathcal{O}_{Z,z}$, and globally one takes the supremum of the local dimensions). 
Locally there is a well-defined notion of irreducible components of an analytic subset, and this notion globalizes by the following theorem of Conrad \cite{Conrad}:
 
 \begin{thm}[Conrad]\label{thm:conrad}
An analytic subset of a rigid analytic space $Y$ admits a unique decomposition into irreducible analytic subsets.
\end{thm}

For the rest of this section, let $B$ be a tropical affine manifold and $\trop: Y(B) \to B$ the corresponding non-Archimedean SYZ fibration as constructed in the previous section.

\begin{defn} If $Z \subset Y(B)$ is an analytic subset, then $\trop(Z) \subset B$ is the \emph{tropicalization} of $Z$.
\end{defn}

If $Z$ is irreducible of dimension $k$, the tropicalization $\trop(Z)$ is expected to carry the structure of a $k$-dimensional tropical subvariety. 
A proof of this result has not appeared in the literature, so we summarize here what is known. 
 
Absent the balancing condition, versions of the result date back to \cite{Bieri1984,Einsiedler2006}. 
A proof in the present setting follows from \cite[Corollary 6.2.2]{Berkovich2004} (see  \cite[Proposition 5.4]{Gubler}):

\begin{thm}\label{thm:gubler}
If $Z \subset Y(B)$ is an irreducible analytic subset of dimension $k$, the tropicalization $\trop(Z) \subset B$ is a locally finite rational polyhedral complex, pure of dimension $k$. 
\end{thm}

A complete proof including weights and balancing can be found in \cite[Theorem 3.3.5]{Maclagan2007}, in the case that $B =\R^n$ and $Z \subset \mathbb{G}_m^n$ is algebraic. 
A complete proof in the case that $Z$ is a curve seems to be ``well-known to experts'', but has only appeared in the literature under certain additional technical hypotheses, cf. \cite{BPR,Yu2015} (in particular, the local case of a curve $Z \subset Y_P$ where $P$ is a simplex is addressed in \cite[Example 1.2]{Yu2015} under the assumption that the coefficient field is discretely valued). 
We shall make the following:

\begin{ass}
\label{ass:an_trop}
If $Z \subset Y(B)$ is an irreducible analytic subset of dimension $1$, the tropicalization $\trop(Z) \subset B$ carries the structure of a tropical curve.
\end{ass}

\begin{cor} 
Suppose that $Z\subset Y(B)$ is an analytic subset which meets a non-empty affinoid subdomain $Y_P = \trop^{-1}(P) \subset Y(B)$ in an analytic curve. Then $B$ contains a tropical curve. 
\end{cor}
\begin{proof}
There is an irreducible component $C$ of $Z$ which contains one of the irreducible components of $Z\cap Y_P$, and is therefore pure of dimension one.  The tropicalization of $C$  is a tropical curve, by Assumption \ref{ass:an_trop}.
\end{proof}

\subsection{Floer theory for cobordisms}\label{Sec:floer_cob}

Let $(X,\omega)$ be a symplectic manifold equipped with an $\infty$-fold Maslov cover. An \emph{unobstructed} Lagrangian brane consists of a Lagrangian brane $L \subset X$ together with an $\omega$-tame almost-complex structure $J$ on $X$ for which $L$ bounds no $J$-holomorphic disc and does not meet any $J$-holomorphic sphere. Any Lagrangian torus fibre $F_b \subset X(B)$ can be equipped with the structure of an unobstructed Lagrangian brane: in fact any $\omega$-tame almost-complex structure $J$ works, since the map $H_1(F_b) \to H_1(X(B))$ is injective and $\omega_{can}$ has trivial integral over any 2-sphere.

An unobstructed cylindrical Lagrangian brane cobordism $V$ consists of a Lagrangian brane cobordism $V\subset \C^*\times X$ together with an almost-complex structure $J$ on $\C^* \times X$ for which $V$ bounds no $J$-holomorphic disc and does not meet any $J$-holomorphic sphere. 
We also require that outside a compact set, $J = J_{\C^*} \oplus J_z$ should decompose as a direct sum, where $J_z$ is a family of almost-complex structures on $X$ parametrized by $z \in \C^*$; and we furthermore require that $J_z$ should coincide with the almost-complex structure $J_i^\pm$ attached to $L_i^\pm$ over the `end' of $V$ modeled on $L_i^\pm$.

\begin{rmk}
\label{rmk:no_taut}
There are few techniques for ensuring unobstructedness of a given cobordism. 
It is for this reason that a more satisfactory treatment would replace the notion of `unobstructed' with `Floer-theoretically unobstructed' everywhere, which means replacing the additional data of almost-complex structures with the additional data of bounding cochains. 
\end{rmk}

\begin{lem}
\label{lem:cob_floer}
Let $(V,J_V)$ be an unobstructed Lagrangian brane cobordism, and $(K,J_K)$ a closed unobstructed Lagrangian brane, both living in $\C^* \times X$. 
Then there is a well-defined Floer cohomology group $HF^*((K,J_K),(V,J_V))$, which we will abbreviate by $HF^*(K,V)$. 

If $(V_t,J_{V_t})_{t \in [0,1]}$ is a one-parameter family of unobstructed Lagrangian brane cobordisms, and $(K_t,J_{K_t})_{t \in [0,1]}$ is a one-parameter family of closed unobstructed Lagrangian branes, such that $V_t$ and $K_t$ are transported  along compactly-supported Hamiltonian isotopies, then there is a continuation isomorphism $HF^*(K_0,V_0) \simeq HF^*(K_1,V_1)$.
\end{lem}
\begin{proof}[Sketch]
By applying a Hamiltonian isotopy to $V$ we may assume that $V$ and $K$ intersect transversely. 
The Floer cochain complex $CF^*(K,V)$ is generated by intersection points between $V$ and $K$. 
In order to define the differential we choose a one-parameter family of $\omega_{\C^*} \oplus \omega$-tame almost-complex structures $J_t$, interpolating between $J_0 = J_V$ and $J_1 = J_K$, and such that outside a compact set the projection map to $\C^*$ is $(J_t,i)$-holomorphic for all $t$. 
The differential then counts $J_t$-holomorphic strips, which means finite-energy solutions of
\begin{align*}
u: \R \times [0,1] & \to \C^* \times X\\
\partial_s u + J_t(u) \partial_t u & =0 \\
u\left(\R \times \{i\}\right) &\subset \left\{ \begin{array}{ll}
											 V, & i=0 \\
											K, & i=1 
										\end{array} \right.
\end{align*}
which are asymptotic to intersection points. 
Compactness of the relevant moduli spaces is achieved by considering the projection of these strips to $\C^*$ \cite[Section 4.2]{Biran-Cornea}, and ruling out disc and sphere bubbling using the unobstructedness assumption (cf. \cite{Seidel:Flux}). 
Continuation maps are constructed in the standard way.
\end{proof}

\subsection{Construction of an analytic subset}
\label{Sec:construct_an_sub}

Let $V \subset \C^*\times X(B)$ be an unobstructed cylindrical Lagrangian brane cobordism between $\bL^- = \{L_i^-\}$ and $\bL^+ = \{L_j^+\}$.   There is a Lagrangian $(n+1)$-torus fibration  
\[
\hat{p}: \C^*\times X(B) \to \R\times B
\]
with fibre over $\hat{b} = (r,b)$ the torus $F_{\hat{b}} = S^1_r \times F_b$, where we have chosen global co-ordinates $(r,\theta)$ on $\C^* = \R\times S^1$. 

The rigid analytic space mirror to this Lagrangian torus fibration is $\mathbb{G}_m \times Y(B)$, which comes equipped with the non-Archimedean SYZ fibration $\trop: \mathbb{G}_m \times Y(B) \to \R \times B$. 
The points $\hat{y} \in \mathbb{G}_m \times Y(B)$ parametrise pairs $F_{\hat{y}} := (F_{\hat{b}},\xi)$ where $F_{\hat{b}}$ is the fibre of $\hat{p}$ over $\hat{b} = \trop(\hat{y})$ and $\xi$ is a $U_\Lambda$-local system on it. 
There is a natural choice of Lagrangian brane structure on $F_{\hat{b}}$, and it is furthermore $H_1$-injective, so any choice of almost-complex structure turns it into an unobstructed Lagrangian brane.
Thus for any such $F_{\hat{y}}$ there is a well-defined $\Z$-graded Floer cohomology group $HF^*(F_{\hat{y}}, V)$, and it is independent of the almost-complex structure assigned to $F_{\hat{y}}$.

\begin{defn} The $i$th \emph{mirror support} $\scrE^i_V$ of $V$ is the subset of $\mathbb{G}_m \times Y(B)$ consisting of points $\hat{y}$  for which the group $HF^i(F_{\hat{y}}, V)$ is non-vanishing. 
The \emph{total mirror support} is $\scrE_V := \cup_i \scrE^i_V$.
\end{defn}

\begin{prop} \label{prop:support_is_analytic} 
The mirror support $\scrE^i_V$ is an analytic subset of $\mathbb{G}_m \times Y(B)$.
\end{prop}

\begin{proof}
This is a consequence of results of Abouzaid  \cite{Abouzaid:ICM} and  Fukaya \cite{Fukaya:cyclic}.
 Fix $V$ and a point $\hat{y_0}$ corresponding to a fibre $F_{\hat{b}_0}$ equipped with unitary local system $\xi_0$. By a small Hamiltonian perturbation, we can assume that $V$ is transverse to the fibres $F_{\hat{b}}$ for $\hat{b}$ in a small polyhedral neighbourhood $P$ of $\hat{b}_0$. 
Thus the Floer cochain complexes $CF^*(F_{\hat{y}},V)$ are identified for all $\hat{y} \in Y_P$, and it is only the Floer differential that varies.  Fix a local section of the family $\{F_b\}$ over $P$, so that for each fibre $F_b$ we have a smoothly varying base-point $\star \in F_{b}$. In $F_{b_0}$, fix a homotopy class of reference path $\gamma_x$ from $x$ to $\star$, for each $x \in V\pitchfork F_{b_0}$; this defines corresponding homotopy classes for each $b\in P$.  

The Floer differential has the shape
\begin{equation} \label{eqn:differential}
m^1(y) \  = \  \bigoplus_x  \sum_{u \in \scrM^{\hat{b}}(x,y)}  q^{E(u)} z^{[\partial u]} \cdot x
\end{equation}
where $x, y$ are intersection points of $F_{\hat{b}}$ and $V$ (which vary constantly over $P$), $\scrM^{\hat{b}}(x,y)$ denotes the set of rigid Floer trajectories, $E(u) = \int u^*\omega$ is the energy of a solution, and $z^{[\partial u]}$ is the monodromy of the local system $\xi$ associated to $\hat{y}$ around a loop $[\partial u]$ obtained from the two boundary arcs of the Floer trajectory $u$ concatenated with the fixed reference paths.

Fix a 1-parameter family of almost complex structures $J_t$ which are appropriate for defining the Floer cohomology group $HF^*(F_{\hat{y}_0}, V)$ as in Lemma \ref{lem:cob_floer}. Following \cite{Fukaya:cyclic}, after shrinking $P$ we may pick a family of $C^2$-small diffeomorphisms $\{\phi_{\hat{b}}\}_{\hat{b} \in P}$ of $\C^* \times X(B)$ with the properties:
\begin{itemize}
\item $\phi_{\hat{b}_0} = \id$;
\item $\phi_{\hat{b}}(F_{\hat{b}}) = F_{\hat{b}_0}$;
\item $\phi_{\hat{b}}$ preserves $V$;
\item $(\phi_{\hat{b}})^*J_t$ tames $\omega_{\C^*} \oplus \omega$ for every $\hat{b} \in P$ and $t \in [0,1]$.
\end{itemize}
For each $\hat{b} \in P$, we define an unobstructed Lagrangian brane cobordism $V_{\hat{b}} := (V, \phi_{\hat{b}}^*J_V)$ which differs from $V$ only in the almost-complex structure. 
By choosing a path $\gamma$ from $\hat{b}$ to $\hat{b}_0$, we obtain a one-parameter family $V_\gamma$ interpolating between $V=V_{\hat{b}_0}$ and $V_{\hat{b}}$; thus by Lemma \ref{lem:cob_floer},  $HF^*(F_{\hat{y}},V) = HF^*(F_{\hat{y}},V_{\hat{b}})$. 
Now our assumptions on $\phi_{\hat{b}}$ ensure that we may define the Floer complexes $CF^*(F_{\hat{y}},V_{\hat{b}})$ by using the almost complex structures $(\phi_{\hat{b}})^*J_t$. 
Pullback by $(\phi_{\hat{b}})^{-1}$  identifies the holomorphic strips contributing to  $CF^*(F_{\hat{y}},V_{\hat{b}})$ with those contributing to $CF^*(F_{\hat{b}_0},V)$. 
 
It follows that, after applying an appropriate rescaling (cf. \cite[Eqn. 3.19]{Abouzaid:ICM}), each term in the Floer differential \eqref{eqn:differential} on $CF^*(F_{\hat{y}},V_{\hat{b}})$ is a `monomial' function of $\hat{y} \in Y_P$. 
Gromov compactness, applied to the holomorphic strips with boundary on $F_{\hat{b}}$, shows that the infinite sum of monomials converges when $\trop(\hat{y}) = \hat{b}$. 
Since this applies for all $\hat{b} \in P$, the coefficients in \eqref{eqn:differential} define analytic functions over $Y_P$, i.e. they belong to the affinoid ring $\mathcal{O}_P$. 

Now $m^1_{\hat{y}}: CF^*(F_{\hat{y}},V_{\hat{b}}) \to CF^{*+1}(F_{\hat{y}},V_{\hat{b}})$ is a linear endomorphism of a fixed finite-dimensional $\Lambda$-vector space, depending analytically on $\hat{y} \in Y_P$. 
The intersection of $\scrE^i_V$ with $Y_P$ is the locus of points where the sum of the ranks of the differentials $CF^{i-1}\to CF^i$ and $CF^i \to CF^{i+1}$ is strictly less than the rank of $CF^i$. 
This is an analytic subset, cut out by the vanishing of appropriate minors of the matrices of these linear maps.
\end{proof}


\begin{prop}\label{prop:support_compactifies}
Suppose $\bL^-$ and $\bL^+$ have multiplicity-one ends. Then 
 $\scrE^i_V$ compactifies to define an analytic subset of $(\mathbb{P}^1)^{an}\times Y(B)$.
 \end{prop}

\begin{proof}
In order to show that the locus compactifies over $0 \in (\mathbb{P}^1)^{an}$, we must show that the coefficients of the Floer differential are analytic over polytope subdomains $\hat{P} = (-\infty,-R] \times P \subset \R \times B$ where the polytopes $P \subset B$ should cover $B$ and are allowed to be arbitrarily small. 
The algebra of functions over such an unbounded domain is 
\[ \mathcal{O}_{\hat{P}} := \left\{ \sum_{a:a_{1} \ge 0}\, f_a \cdot z^a \ : \ val(f_a) + a \cdot \hat{p} \to \infty \ \textrm{as} \ \|a\| \to \infty \text{ for all $\hat{p} \in \hat{P}$} \right\}\]
(see \cite[Section 6]{Rabinoff} -- in general the only powers $a$ which are allowed to appear are those such that $a \cdot \hat{p}$ is bounded above on the region $\hat{P}$, which for $\hat{P} = (-\infty,-R] \times P$ is equivalent to $a_{1} \ge 0$). 
So we must show that, under our hypotheses (on the shape of the almost complex structure etc)  the coefficients of the Floer differential will lie in such an algebra of functions, for torus fibres $F_{\hat{y}}$ sufficiently far out in the negative $\R$ direction.

\begin{figure}[ht]
\begin{center} 
\begin{tikzpicture}
\draw[thick, blue, dashed] (-0.75,-2.5) -- (-0.75,2.5); 
\draw[thick, blue, dashed] (0.75,-2.5) -- (0.75,2.5); 
\draw[semithick] (-3,-2.5) -- (-3,2.5);
	\draw (-3,2.9) node {$\widetilde{S^1_r} \times F_b$};
\draw[semithick] (-4,-1.95) -- (-.75,-1.95);
	\draw (-4.6,-1.95) node {$\sigma^2(\tilde{L}_1^-)$};
\draw[semithick] (-4,-.95) -- (-.75,-0.95);
	\draw (-4.6,-.95) node {$\sigma(\tilde{L}_2^-)$};
\draw[semithick] (-4,0.05) -- (-.75,0.05);
	\draw (-4.6,0.05) node {$\sigma(\tilde{L}_1^-)$};

\draw[semithick] (-4,1.05) -- (-.75,1.05);
	\draw (-4.5,1.05) node {$\tilde{L}_2^-$};

\draw[semithick] (-4,2.05) -- (-.75,2.05);
	\draw (-4.5,2.05) node {$\tilde{L}_1^-$};

\draw[semithick] (0.75,1.8) -- (4,1.8);
	\draw (4.5,1.8) node {$\tilde{L}^+$};

\draw[semithick] (0.75,-.2) -- (4,-.2);
	\draw (4.6,-0.2) node {$\sigma(\tilde{L}^+)$};

\draw[semithick] (0.75,-2.2) -- (4,-2.2);
	\draw (4.6,-2.2) node {$\sigma^2(\tilde{L}^+)$};
	
\fill[gray!10] (-0.75,-2.5) -- (-0.75,2.5) -- (0.75,2.5) -- (0.75,-2.5) -- cycle;
\draw (0,0) node {$\tilde{V}$};

\draw[fill=black!20!white] (-3,2.05) -- (-0.75,2.05) -- (-0.75,1.05) -- (-3, 1.05) -- cycle;
\draw[fill=black!20!white] (-3,1.05) -- (-0.75,1.05) -- (-0.75,0.05) -- (-3, 0.05) -- cycle;
	\draw[fill] (-3,2.05) circle (0.1); \draw (-3.2,1.8) node {$x$}; 	
	\draw[thick, arrows = ->] (-2.8, 1.9) .. controls (-2.5,1.5) and (-2.5,1) .. (-2.8, 0.2);
		\draw (-1.9,1.4) node {$m^1_1(x)$};

\end{tikzpicture}
\end{center}
\caption{The $q_{1}$-filtration lifted via the cover $\C\to\C^*$, and the image of a disc of weighted by $q_{1}^1$.}
\label{Fig:lifts}
\end{figure}

We can choose $R \gg 0$ such that the cobordism $V$ is cylindrical over the region $|r| \ge R-1$.  Let $\tilde{V}$ be the preimage of the cobordism $V$ in the infinite cyclic cover $\C \times X(B)$ of $\C^*\times X(B)$.  Each end $L_i^\pm$ of $V$ has infinitely many lifts, differing by  deck transformations. Let $\sigma$ denote the generating deck transformation. The Floer differential can be expanded as 
\begin{equation} \label{eqn:sum_over_winding}
m^1 = \sum_{j=-\infty}^\infty m^1_j(z_2,\ldots,z_{n+1}) \cdot z_{1}^j,
\end{equation}
where $m^1_j$ counts Floer strips whose boundary component on $F_{\hat{b}}$ winds $j$ times around the base cylinder $S^1_r$, i.e. whose lift connects a fixed lift $\tilde{L}_1^-$ to its image $\sigma^j(\tilde{L}_1^-)$ shifted $j$ places by the deck group, cf. Figure \ref{Fig:lifts}. 
Note that there are no Floer strips with fractional winding, by our assumption that $\bL^\pm$ have multiplicity-one ends (compare Remark \ref{rmk:fibre_mult_higher}). 
The lift of such a holomorphic strip projects to a holomorphic strip in $\C$ with one boundary on the universal cover $\tilde{S}^1_r$ of $S^1_r$ and the other lying in the projection of $\tilde{V}$. It is clear from the open mapping theorem that the $\tilde{S}^1_r$-boundary of such a strip can only move in one direction, so $m^1_j = 0$ for $j<0$ (compare to the proof of upper triangularity of the Floer differential in \cite[Lemma 4.4.1]{Biran-Cornea}). 
Therefore the only powers $z^a$ that can appear in the Floer differential have $a_{1} \ge 0$ as required.

It remains to check convergence of the Floer differential. 
We can do this by the Fukaya trick: the fact that the cobordism is cylindrical at infinity allows us to make the Fukaya trick work over the larger polytope.  
Let $J_t$ be a one-parameter family of $(\omega_{\C^*}\oplus \omega)$-tame almost-complex structures, appropriate for defining the Floer cohomology group $HF^*(S^1_{-R} \times F_b,V)$ as in Lemma \ref{lem:cob_floer}. 
We assume that $J_t=J_{\C^*} \oplus J_{z,t}$ for $r \le -R+1$. 

Now, we take a family of diffeomorphisms $\phi_s: \C^* \to \C^*$ for $s \ge 0$, having the form $\phi_s(r,\theta) = (\psi_s(r),\theta)$ where $\psi_s(r)$ is equal to the identity for $r \ge -R+1$ and $\psi_s(-R) = -R-s$.
In particular we note that $\phi_s$ takes $S^1_{-R}$ to $S^1_{-R-s}$. 
It is clear that if $J_{\C^*}$ is the standard complex structure on $\C^*$, and $\omega_{\C^*} = dr \wedge d\theta$ the standard symplectic form, then $(\phi_s)_* J_{\C^*}$ tames $\omega_{\C^*}$ for any $s$.
Therefore $J_{s,t} = (\phi_s \times \id)_*(J_{\C^*} \oplus J_{z,t})$ is a family of $(\omega_{\C^*} \oplus \omega)$-tame almost complex structures, which we can use to define the differential in $CF^*(S^1_{-R-s} \times F_b,V)$.
 
As in the usual Fukaya trick, the moduli spaces of strips defining the differentials in these Floer cohomology groups can be identified for different values of $s \ge 0$, but they are counted with different weights. 
Combining with the Fukaya trick in the direction of the base $B$, it follows that the coefficients of the Floer differential are convergent over $(-\infty,-R] \times P \subset \R \times B$ for sufficiently small $P$. 
It follows that $\scrE^i_V$ compactifies over $0 \in (\bP^1)^{an}$; the argument over $\infty \in (\bP^1)^{an}$ is identical.
\end{proof}

\begin{rmk}\label{rmk:fibre_mult}
The fibre of the compactification of $\scrE^i_V$ over $0 \in (\bP^1)^{an}$ is the locus of points $(0,y) \in (\bP^1)^{an} \times Y(B)$ such that $HF^i(F_y,L_j^-) \neq 0$ for some negative end $L_j^-$ of $V$ (as is clear from the proof of Proposition \ref{prop:support_not_empty}). 
Thus it is the union of the $i$th mirror supports of the negative ends of $V$. 
Similarly, the fibre over $\infty$ is the union of the $i$th mirror supports of the positive ends of $V$. 
\end{rmk}

\begin{rmk} \label{rmk:fibre_mult_higher}
Proposition \ref{prop:support_compactifies} generalises to the case in which $\bL^{\pm}$ may have higher-multiplicity ends.  Suppose that  $V$ has $a$ negative and $b$ positive ends, not necessarily carrying distinct torus fibres.  In this case, there may be Floer solutions which project to have  `fractional' winding in the base (cf. Figure \ref{Fig:lifts}, and imagine that $L_1^- = L_2^-$), and it is  therefore natural to expand the Floer differential in powers of $z_1^{1/a}$ near the negative end, and $z_1^{1/b}$ near the positive end. The mirror support most naturally compactifies to define an analytic subset of $\bP^1(a,b)^{an} \times Y(B)$, where $\bP^1(a,b)^{an}$ is the analytification of the weighted projective line (see \cite{Ulirsch} for a construction of  rigid analytic spaces associated to toric stacks; $\bP^1(a,b)^{an}$ can be constructed by hand by gluing charts).  There is a pushforward map $\bP^1(a,b)^{an} \to (\bP^1)^{an}$ coming from the forgetful map from the stack to the coarse moduli space. The push-forward of the mirror support is still analytic. Indeed, if  the mirror support is cut out near the negative end by analytic functions $p\left(z_1^{1/a},z_2,\ldots,z_{n+1}\right)$, then the push-forward is cut out by the analytic functions
\[
f(z_1,\ldots,z_{n+1}) = \prod_{\omega: \omega^a = 1} p\left(\omega z_1^{1/a},\ldots, z_{n+1}\right).
\]
To explain this equation, note that if $p$ were a polynomial in $z_1^{1/a}$, then the product on the right-hand side could be written as a polynomial in $z_1$ by standard properties of the resultant. 
The same result applies to analytic functions, by applying the result for polynomials to the polynomial truncations and passing to the limit. 
\end{rmk}

\subsection{Non-emptiness for distinct ends}

As before, let $V \subset \C^*\times X(B)$ be an unobstructed Lagrangian brane cobordism between tuples of fibres $\bL^- = (F_{b_i^-})$ and $\bL^+ = (F_{b_i^+})$.
It is clear that $\trop(\scrE_V)$ is contained in the image of $V$ under projection to $\R \times B$ (since two non-intersecting Lagrangians have vanishing Floer homology).   
In particular the tropicalization is contained in a compact set union the negative ends $\R_- \times  b_i^-$ and the positive ends $\R_+ \times b_i^+$. 
The following gives a condition under which $\trop(\scrE_V)$ contains such an end (it is phrased for negative ends but applies equally to positive ones):

\begin{prop} \label{prop:support_not_empty}
If a fibre $F_b$ in $\bL^-$ has multiplicity one, then $\trop(\scrE^0_V)$ contains an end $(-\infty,-R) \times b$ of multiplicity one.
\end{prop}

\begin{proof} 
Suppose that $(-R-1,b)$ is contained in a polytope domain  $\hat{P} = (-\infty,-R] \times P \subset \R \times B$ over which the Floer differential on $CF^*(F_{\hat{y}},V)$ converges, as in the proof of Proposition \ref{prop:support_compactifies}. If $\trop(\hat{y})=\hat{b}=(r,b)$ with $r \le -R-1$, then $F_{\hat{b}} \cap V = F_b$ is a clean intersection comprising a  single copy of the torus $F_b$ itself.  
Thus we can use a Morse--Bott model (cf. \cite{Seidel:Flux}) for the Floer cochain group $CF^*(F_{\hat{y}},V) = C^*_{Morse}(F_b) \simeq \wedge^*(\Lambda^n)$, where we take a perfect Morse function on $F_b \simeq T^n = S^1\times \cdots\times S^1$ which is a sum of Morse functions on the distinct $S^1$-factors. 
Since $CF^{-1}=0$ and $CF^0$ has rank one and is spanned by the identity $e \in C^0_{Morse}$, in order to determine $\scrE^0_V$, we only need to consider the differential $m^1(e) \in \Lambda^n$.

We equip $V$ with the trivial local system, and let $(z_1,\ldots,z_{n+1})$ be the monodromies of the local system on $F_{\hat{y}}$ around the different $S^1$ factors. 
We can expand the Floer differential applied to $e$ as
\[ m^1(z_1,\ldots,z_{n+1}) = \sum_{j \ge 0} m^1_j(z_2,\ldots,z_{n+1})\cdot z_{1}^j\]
as in the proof of Proposition \ref{prop:support_compactifies}, where we recall that $m^1_j$ counts strips that wind $j$ times around the $S^1_r$ factor.

The zero-order term $m^1_0$ counts strips with no winding. The open mapping theorem applied to the image of the disc projected to $\C$ shows that any such Floer strip must have constant projection and hence live inside a single fibre, so we have
\begin{equation} \label{eqn:lowest_order}
m^1_0(z_2,\ldots,z_{n+1}) = (z_2-1,z_3-1,\ldots, z_{n+1}-1)
\end{equation}
by considering the parallel transport of the local system along the two arcs of each circle factor (Morse trajectories) in $F_b$. 
Thus we have a unique solution $(0, 1,\ldots,1)$ to $m^1(z_1,\ldots,z_{n+1})=0$, inside $\{0\} \times Y_P$; and furthermore, the matrix of partial derivatives $\{\partial m^1 / \partial z_i\}_{2 \le i \le n+1}$ is invertible at this solution. 
By the implicit function theorem \cite[(10.8)]{Abhyankar1964}, there exists $\delta>0$ and power series $\{z_i(z_{1})\}_{2\le i\le n+1}$ with $z_i(0)=1$, convergent for $|z_{1}|<\delta$, such that $(z_2(z_{1}),\ldots,z_{n+1}(z_{1}))$ is the unique solution of the `equation' $m^1(z_1,\ldots,z_n,z_{n+1})=0$ in a neighbourhood of $(1,\ldots,1)$. 
It follows that $\trop(\scrE^0_V)$ contains the edge $(-\infty,\log(\delta)) \times \{b\}$ with weight $1$.
\end{proof}

\begin{rmk} \label{rmk:common_ends}
Note that if the ends of $V$ are not geometrically distinct, this argument breaks down. Suppose $V$ has $k \geq 2$ ends modelled on $F_b$, so $CF^*(F_{\hat{b}}, V) \ = \ \oplus_{j=1}^k H^*(F_b;\Lambda)$. 
If $e_j$ denotes the identity element in the $j$-th copy of $H^*(F_b;\Lambda)$, we  have
\[
m^1_0\left(\bigoplus_{j=1}^k a_j \cdot e_j\right) = \bigoplus_{j=1}^k a_j \cdot (z_2-1,\ldots, z_{n+1}-1).
\]
It is no longer true that the matrix of partial derivatives $\{\partial m^1_0/\partial z_i\}_{2 \le i \le n+1}$ is invertible.
\end{rmk}

\subsection{Separating ends\label{Sec:separable}}

In view of Remark \ref{rmk:common_ends}, if all ends of $V$ have multiplicity greater than one, it is not clear that $V$ has non-trivial Floer cohomology with any torus fibre in $\C^*\times X(B)$ or that $\scrE_V$ is non-empty. Under additional topological hypotheses, one can ensure this by geometrically separating the ends of $V$ by a Lagrangian isotopy. 

Let $V \subset \C^*\times X(B)$ be a Lagrangian brane cobordism between fibres, with ends $L_i^\pm$ having multiplicities $m_i^\pm$. 
For any $\alpha \in H^1(V;\R)$ and sufficiently small $\varepsilon \in \R$ we can deform $V$ by a symplectic isotopy of flux $\varepsilon \cdot \alpha$, to obtain a new Lagrangian submanifold $V_{\epsilon \cdot \alpha}$. 
We can arrange that the deformation is also a cobordism, i.e., it also has cylindrical ends which are fibres.

\begin{lem} \label{lem:how_to_separate}
Suppose that the classes $\iota_j^* \alpha \in H^1(L_i^\pm;\R)$ are pairwise distinct for $1 \le j \le m_i^\pm$. Then for $\epsilon \neq 0$, the cobordism $V_{\epsilon \cdot \alpha}$ has $m_i$ ends of multiplicity one which are pairwise distinct perturbations of $L_i$.
\end{lem}
\begin{proof}
The copies of $L_i$ are shifted by Lagrangian isotopies with the pairwise distinct fluxes $\iota_j^* \alpha$.
\end{proof}

\begin{defn}
We say that the end $L_i^\pm$ is `separable' if the maps
\begin{equation} \label{eqn:separable}
\iota_j^*:H^1(V;\R) \to  H^1(L_i^\pm;\R)  \qquad 1 \leq j \leq m_i^\pm
\end{equation}
induced by inclusion are pairwise distinct; and that $V$ has `separable ends' if all of its ends are separable.
\end{defn}

If the end $L_i^\pm$ is separable, then the classes $\iota_j^* \alpha$ are pairwise distinct for $1 \le j \le m_i^\pm$ for all $\alpha$ in the complement of a finite union of proper linear subspaces. 
Thus the end $L_i^\pm$ is `separated' into distinct ends of multiplicity one in $V_{\varepsilon \cdot \alpha}$, for generic $\alpha$. 
If $V$ has separable ends, then all ends of $V$ are separated into distinct ends of multiplicity one in $V_{\varepsilon \cdot \alpha}$ for generic $\alpha$.

\begin{cor} \label{cor:separable_tropical}
If $V$ is unobstructed and has separable ends, then $\R\times B$ contains a tropical curve whose ends are given by the ends of $V$, weighted with the multiplicities of the corresponding Lagrangians.
\end{cor}
\begin{proof}
We consider the non-Archimedean SYZ fibration $\trop: H^1(V;\Lambda^*) \to H^1(V;\R)$ induced by $val:\Lambda^* \to \R$. 
For a small polytope $P$ enclosing the origin in $H^1(V;\R)$, we consider the polytope subdomain $U_P := \trop^{-1}(P)$. 
The points $u \in U_P$ parametrize pairs $V_u = (V_{\trop(u)},\xi_u)$ consisting of a Lagrangian deformation of $V$ with flux $\trop(u)$, equipped with a $U_\Lambda$-local system $\xi_u$. 

We consider the set of points $(\hat{y},u) \in \mathbb{G}_m \times Y(B) \times U_P$ such that $HF^i(F_{\hat{y}},V_u) \neq 0$. 
Applying the Fukaya trick as in the proof of Proposition \ref{prop:support_is_analytic}, we find that this is an analytic subset $\EuF^i$. 
Furthermore, since $V$ has separable ends, $V_u$ has distinct ends so long as $\trop(u)$ lies off of a finite union of proper linear subspaces. 
The arguments of the previous section show that, locally near one of these distinct ends, the projection $\EuF^0 \to \mathbb{G}_m \times U_P$ is an isomorphism. We take the union of irreducible components containing these ends, which defines a subset of $\EuF^0$ which has dimension $1+b_1(V)$.  
The tropicalization of this union of components gives a locally-finite rational polyhedral complex $\Gamma \subset \R \times B \times P$, pure of dimension $1+b_1(V)$, by Theorem \ref{thm:gubler}. 

Let us choose an interval $I \subset P$, passing through the origin and not lying in the finite union of proper linear subspaces referenced above. 
We define $\Gamma_I := \Gamma \cap \R \times B \times I$.
Shrinking $I$ if necessary, and using the local-finiteness of $\Gamma$, we may assume that the fibre over any $t \in I \setminus \{0\}$ is a one-dimensional rational polyhedral complex $C_t \subset \R \times B$ whose ends coincide with the ends of $V_t$ outside of a fixed compact subset of $\R \times B$, and have multiplicity one. We can show that $C_t$ has the structure of a tropical curve by Assumption \ref{ass:an_trop}. 

By the local finiteness of $\Gamma_I$, for any point in $\R \times B$ there exists a neighbourhood $U$ of the point and an $\epsilon > 0$, such that the topology of $C_t \cap U$ is constant for $t \in (0,\epsilon)$ and the edges and vertices vary linearly. 
Thus there is a well-defined limit of each vertex and edge as $t \to 0^+$. 
We define the tropical curve $C_0$ to be the union of all limits of vertices and edges of $C_t$. 
We define the weight associated to an edge of $C_0$ to be the sum of weights of edges that converge to that edge. 
Now consider a vertex $v$ of the limiting curve $C_0$. 
Consider a ball in $\R \times B$ which is centered on $v$, and small enough that it contains no other vertex of $C_0$. 
For sufficiently small $t \neq 0$, the ball only contains vertices of $C_t$ that converge to $v$. 
The sum of the primitive vectors pointing along edges of $C_t$ exiting this ball, multiplied by their weights, vanishes because $C_t$ is balanced (the contributions coming from edges contained entirely inside the ball cancel out). 
It follows that $C_0$ is balanced at $v$. 
Therefore the limit $C_0$  is a tropical curve, which has the required properties by construction.
\end{proof}

\begin{rmk}
If we were working with Floer-theoretically unobstructed cobordisms rather than tautologically unobstructed ones, we would additionally need to show that the bounding cochains persisted under the perturbations of $V$ introduced in the proof of Corollary \ref{cor:separable_tropical}.
\end{rmk}

\subsection{Conclusions} Let $B$ be a tropical affine torus which contains no tropical curve. The following is immediate from Corollary \ref{cor:separable_tropical}:

\begin{prop}\label{prop:nosepcob}
If $V\subset \C^*\times X(B)$ is an unobstructed Lagrangian brane cobordism between tuples of fibres $\bL^-$ and $\bL^+$, then either $\bL^-=\bL^+$ or $V$ does not have separable ends.
\end{prop}

Theorem \ref{thm:none} is an immediate consequence, because ends of multiplicity one are separable ends. 
Combining with Theorem \ref{thm:formlift}, we obtain:
%

\begin{cor}
Suppose that the tropical affine torus $B$ contains no tropical curve, but does contain a mixed-sign tropical curve. Then there exists an embedded formal Lagrangian brane cobordism in $X(B)$, which admits an immersed Lagrangian brane representative, but which admits no unobstructed representative.
\end{cor} 

Recall that $\Cob_{\fib}^\unob(X(B)/\C^*)$ denotes the free abelian group generated by torus fibres of $X(B) \to B$, equipped with their natural structure of unobstructed Lagrangian branes, modulo relations coming from cylindrical unobstructed Lagrangian brane cobordisms all of whose ends are fibres.

\begin{conj} \label{conj:free} If $B$ contains no tropical curve, then $\Cob_{\fib}^\unob(X(B)/\C^*) \simeq \Z^{B}$.
\end{conj}

The conjecture is consistent with expectations from mirror symmetry, cf. Corollary \ref{cor:Chow_free} and Section \ref{Sec:cobrat}.

\section{Rational equivalence} 
\label{Sec:cob_k}

This section outlines a relationship between cylindrical Lagrangian cobordism in $X(B)$ and rational equivalence of cycles in the mirror $Y(B)$.  Whilst these results provide motivation and context for our main results, they are not required for the proofs, so our treatment will be somewhat less detailed than in earlier sections.

\subsection{K-theory of the Fukaya category}

Recall that the Grothendieck group of a triangulated category is the free abelian group generated by its objects, modulo relations $[C] = [B] - [A]$ whenever $C \simeq Cone(A \to B)$. 
If $\scrA$ is a triangulated $A_\infty$ category then $H^0(\scrA)$ is a triangulated category (see \cite[Chapter 3]{Seidel:FCPLT}), so we can define $K(\scrA) := K\left(H^0(\scrA)\right)$.
Given an $A_\infty$ category $\scrA$, there are various natural formal enlargements of it which are triangulated (cf. \cite[Section 1]{Seidel:Flux} and \cite[Lecture 7]{Seidel:Cat_Dyn}):
\[ \scrA^{tw} \subset \scrA^{perf} \subset \scrA^{mod} \supset \scrA^{prop}.\]
The `smallest' is the category of twisted complexes $\scrA^{tw}$, after which we have its split closure $\scrA^{perf}$. 
These both sit inside the category $\scrA^{mod}$ of left $A_\infty$ $\scrA$-modules. 
Also sitting inside $\scrA^{mod}$ we have the subcategory $\scrA^{prop}$ of proper $\scrA$-modules, i.e., those with finite-dimensional cohomology for each object. 
The inclusion functors between these categories induce maps between the corresponding Grothendieck groups
\[ K\left(\scrA^{tw}\right) \hookrightarrow K\left(\scrA^{perf}\right) \to K\left(\scrA^{mod}\right) \leftarrow K\left(\scrA^{prop}\right).\]
We note that the leftmost arrow is an injection by \cite[Corollary 2.3]{Thomason1997}.

If $\scrA$ is itself proper (i.e., all its $hom$-spaces have finite-dimensional cohomology) then $\scrA^{perf} \subset \scrA^{prop}$. 
If $\scrA$ is smooth (i.e., the diagonal bimodule is perfect) then $\scrA^{prop} \subset \scrA^{perf}$. 
In particular, if $\scrA$ is saturated (i.e., both proper and smooth) then $\scrA^{prop} = \scrA^{perf}$. 

Now let $(X,\omega)$ be a closed symplectic manifold equipped with an $\infty$-fold Maslov cover. The Fukaya category $\scrF(X)$ is a $\Z$-graded $A_{\infty}$ category linear over the Novikov field, whose objects are unobstructed Lagrangian branes. 
It is proper (since two Lagrangian submanifolds generically intersect in a finite number of points), and in good situations it is expected that $\scrF(X)^{perf}$ should furthermore be smooth. 
For example, if $X$ admits a homological mirror $Y$ which is smooth and proper, then $\scrF(X)^{perf} \simeq \EuD (Y)$ is smooth and proper (here `$\EuD(Y)$' denotes a dg enhancement for the bounded derived category of coherent sheaves on $Y$). 
In this case we have associated Grothendieck groups
\[ K\left(\scrF(X)^{tw}\right) \subset K\left(\scrF(X)^{perf} \right) = K\left(\scrF(X)^{prop}\right) .\]

\subsection{K-theory relations from planar cobordisms}\label{Sec:K0_planar}

Biran and Cornea have proved that planar Lagrangian brane cobordisms induce relations between their ends in the Grothendieck group of the Fukaya category. 
In particular,

\begin{thm}[Biran-Cornea]\label{thm:BC-K0-rel}
There is a surjective homomorphism
\begin{align*}
\Cob^\unob(X/\C) & \to K\left(\scrF(X)^{tw}\right), \quad \text{sending} \\
[L] & \mapsto [L].
\end{align*}
\end{thm}

Strictly, \cite{Biran-Cornea-2} proves this result for ungraded monotone cobordisms of monotone Lagrangians, but the underlying geometric arguments should apply with minor modifications to the present case. 

In fact, Biran and Cornea prove something stronger than Theorem \ref{thm:BC-K0-rel}. 
Observing that any positive end in a planar cobordism can be `turned around' into a negative end, they restrict their attention to cobordisms with a single positive end $L^+$ and negative ends $L_1^-,\ldots,L_n^-$. 
They prove that such a cobordism induces an iterated cone decomposition
\begin{equation}\label{eqn:itcone}
 L^+ \simeq L_1^- \to \left(L_2^- \to \left(L_3^- \to \ldots \right) \right)
 \end{equation}
in $\scrF(X)^{tw}$ \cite[Theorem A]{Biran-Cornea-2}. 
We feel it may be useful to the reader to see a sketch proof of this result in the language of `partially wrapped' Fukaya categories \cite{Sylvan}, since it will help to understand the difference between planar and cylindrical cobordism. 

Consider the Liouville manifold $\C$, compactified by an ideal boundary to a disc, and let $\sigma = \{\pm i \infty\}$ be a set of two `stops' on the ideal boundary. 
Unobstructed planar cobordisms define objects of the partially wrapped Fukaya category $\scrW(\C \times X,\sigma \times X)$. 
The latter can be defined by combining Seidel's approach to defining Floer cohomology for unobstructed Lagrangian branes \cite{Seidel:Flux} with any of the approaches to defining partially wrapped Fukaya categories \cite{Biran-Cornea,Sylvan,Seidel2018,Ganatra2017}. 
In our conventions, the $hom$-spaces in the partially wrapped Fukaya category are $hom^*(V,V') := CF^*(\phi_\varepsilon(V),V')$, where the flow $\phi_{\varepsilon}$ wraps $V$ towards the stops at the boundary (but not past them), and is chosen so as to shift all heights of $V$ past those of $V'$, in particular removing intersections near infinity. 
In the case that $V' = K$ is closed, $HF^*(\phi_{\varepsilon}(V),K) \simeq HF^*(V,K)$ coincides with the Floer cohomology group defined in Lemma \ref{lem:cob_floer}, and is independent of conventions for wrapping at infinity.

By the K\"unneth theorem for partially wrapped Fukaya categories \cite[Theorem 1.5]{Ganatra2018}, there is an $A_\infty$ bifunctor
\begin{align*}
 i:\scrW(\C,\sigma) \otimes \scrF(X) & \hookrightarrow \scrW(\C \times X,\sigma \times X)\\
 i(\gamma, K) &\coloneqq \gamma \times K
 \end{align*}
which is an embedding on the level of cohomology. 
Thus, given an object $\gamma$ of $\scrW(\C,\sigma)$ and a planar cobordism $V$ defining an object of $\scrW(\C \times X,\sigma \times X)$, we obtain an $A_\infty$ functor to the category of chain complexes as the composition:
\[ \scrF(X) \xrightarrow{i(\gamma,-)} \scrW(\C \times X, \sigma \times X) \xrightarrow{hom^*(-,V)} \mathsf{Ch}.\]
This defines an $A_\infty$ module over $\scrF(X)$, which we denote by $\scrY(\gamma,V) \in \ob \scrF(X)^{mod}$.

Let us take $\gamma = \R$ to be the real axis in $\C$. 
We choose sufficiently wrapped Hamiltonian-isotopic paths $\gamma', \gamma''$ as in Figure \ref{Fig:K0_relation_over_C}, giving us two cochain-level models for the module $\scrY(\gamma,V)$:
\[ \scrY(\gamma,V)(K) \simeq CF^*(\gamma' \times K,V) \simeq CF^*(\gamma'' \times K,V)\]
(these two models are quasi-isomorphic by the Hamiltonian isotopy invariance of Floer cohomology).
 
One then observes that there is an isomorphism of modules $CF^*(\gamma'' \times K,V) \simeq CF^*(K,L^+)$: all intersection points between $\gamma'' \times K$ and $V$ lie in a single fibre of the projection $\C \times X \to \C$, and all holomorphic discs must lie in the same fibre by the open mapping theorem applied to the projection to $\C$. 
There is a similar isomorphism of vector spaces $CF^*(\gamma' \times K,V) \simeq \oplus_j CF^*(K,L_j^-)$, but the structure maps need no longer count holomorphic discs inside a single fibre: nevertheless one can prove that they are lower-triangular, by considering the projections of the holomorphic discs to $\C$ as in the proof of Proposition \ref{prop:support_not_empty}. 
For example, the holomorphic strips contributing to the component of the differential mapping $CF^*(\gamma' \times K,L_1^-) \to CF^*(\gamma' \times K,L_2^-)$ will project to the unique compact region bounded by $\gamma'$ and the projection of $V$ in Figure \ref{Fig:K0_relation_over_C}.
It follows that we have the desired quasi-isomorphism of Yoneda modules
\[ \scrY(L^+) \simeq \scrY(\gamma,V) \simeq \scrY\left(L_1^-\right) \to \left(\scrY\left(L_2^-\right) \to \left(\scrY\left(L_3^-\right) \to \ldots \right) \right).\]

\begin{figure}[ht]
\begin{center} 
\begin{tikzpicture}
		\draw (-0.95,1.5) node {$L_1^-$};
		\draw (-0.95,0.9) node {$L_2^-$};

		\draw (6.2,1.7) node {$L^+$};
		\draw[fill, color=red] (1,0.2) circle  (0.05);

	\draw[fill] (1,0.9) circle  (0.07);
	\draw[dashed,color=gray] [ ->] (5.6,0) -- (4.6,1.6) ;
	\draw (6,-0.4) node {$CF^*(K, L^+)$};
	\draw[dashed,color=gray] [ ->] (-0.4,3) -- (0.85,1.6) ;
	\draw (-0.5,3.4) node {$CF^*(K, L_1^-)$};
		\draw[fill] (1,1.5) circle  (0.07);
	\draw[fill] (4.5,1.7) circle (0.07);
	\draw (3,3.2) node {$\gamma' \times K$};
	\draw (3,-0.2) node {$\gamma'' \times K$};
	\draw[semithick, color=red, rounded corners] (-0.48,0.2) -- (4.5,0.2) -- (4.5, 2.7) -- (5.7, 2.7);
	\draw[semithick, color=red, rounded corners] (-0.48,-0.1) -- (1,-0.1) -- (1,2.85) -- (5.7, 2.85);
	\draw[semithick] (-0.48,1.5) -- (2,1.5);
	\draw[semithick] (-0.48,0.9) -- (2.2,0.9);
	\draw[semithick] (3.97,1.7) -- (5.7,1.7);
	\draw[semithick] (3,1.5) circle(1cm) node{$V$};
	\draw[thick,arrows=->] (-2.6,3.5) -- (-2.6,-.5); 
	\draw (-3,1.5) node {$\phi_{t}$};
	\draw[thick,arrows=->] (7.7,-.5) -- (7.7,3.5); 
	\draw (8,1.5) node {$\phi_{t}$};

\end{tikzpicture}
\end{center}
\caption{Consequences of the cone decomposition for Lagrangian Floer cohomology.}
\label{Fig:K0_relation_over_C}
\end{figure}
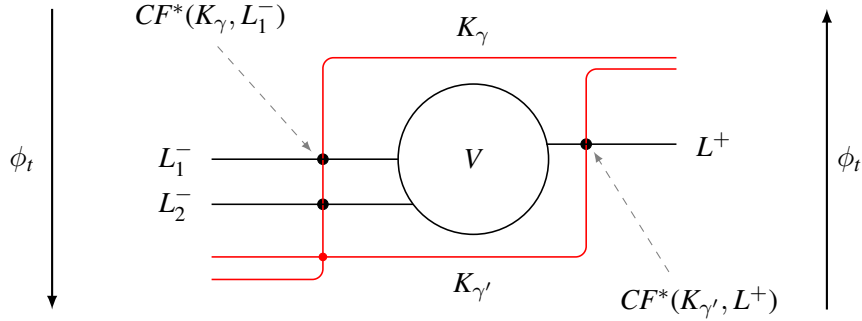

\subsection{K-theory relations from cylindrical cobordisms \label{Sec:K0_cyl}}

We consider the Liouville manifold $\C^*$, compactified by an ideal boundary to an annulus, and let $\sigma$ consist of a single stop on each ideal boundary component, located on the negative real locus. 
Unobstructed cylindrical cobordisms define objects of the partially wrapped Fukaya category $\scrW(\C^* \times X, \sigma \times X)$. 
As before, an object $\gamma$ of $\scrW(\C^* \times X,\sigma \times X)$ together with a cylindrical cobordism $V$ determines a module $\scrY(\gamma,V) \in \ob \scrF(X)^{mod}$. 
In fact $\scrY(\gamma,V) \in \ob \scrF(X)^{prop}$, because $\varphi_\epsilon(\gamma \times K) \cap V$ is generically finite.

The arguments sketched in the previous section require some modification. 
For example, if we take $\gamma$ to be the positive real locus, and consider sufficiently wrapped Hamiltonian-isotopic paths $\gamma',\gamma''$ as in Figure \ref{Fig:K0_relation_over_Cstar}, then there are some extra intersection points between $\gamma' \times K$ and $V$ arising from the region where $V$ wraps around the back of the cylinder.
Nevertheless we have:

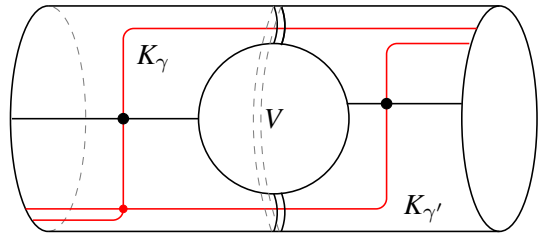
\begin{figure}[ht]
\begin{center} 
\begin{tikzpicture}
	\draw[dashed,color=gray] (0,0) arc (-90:90:0.5 and 1.5);
	\draw[semithick] (0,0) -- (6,0);
	\draw[semithick] (0,3) -- (6,3);
	\draw[semithick] (0,0) arc (270:90:0.5 and 1.5);
	\draw[semithick] (6,1.5) ellipse (0.5 and 1.5);
	\draw[semithick, color=red, rounded corners] (-0.3,0.3) -- (4.5,0.3) -- (4.5, 2.5) -- (5.6, 2.5);
	\draw (5.05,0.25) node {$\gamma'' \times K$};
	\draw[semithick, color=red, rounded corners] (-0.2,0.15) -- (1,0.15) -- (1,2.7) -- (5.7, 2.7);
	\draw (1.63,2.4) node {$\gamma' \times K$};
			\draw[fill, color=red] (1,0.3) circle  (0.05);
		\draw[fill] (1,1.5) circle  (0.07);
	\draw[fill] (4.5,1.7) circle (0.07);
	\draw[semithick] (-0.48,1.5) -- (2,1.5);
	\draw[semithick] (3.97,1.7) -- (5.5,1.7);
	\draw[semithick] (3,1.5) circle(1cm) node{$V$};
	\draw[semithick]  (3,2.5) arc (-21:21:0.7) ;
	\draw[semithick]  (3.1,2.48) arc (-21:21:0.72) ;
	\draw[semithick]  (3.1,2.48) arc (-21:21:0.72) ;
	\draw[semithick]  (3.1,0) arc (-21:21:0.72);
	\draw[semithick]  (3,0) arc (-21:21:0.7) ;
	\draw[dashed,color=gray] (3.1,3) arc (160:200:4.5) ;

	\draw[dashed,color=gray] (3,3) arc (160:200:4.5) ;
	
\end{tikzpicture}
\end{center}
\caption{Extra intersection points between $\gamma \times K$ and $V$.}
\label{Fig:K0_relation_over_Cstar}
\end{figure}

\begin{figure}[ht]
\begin{center} 
\includegraphics[width=0.6\textwidth]{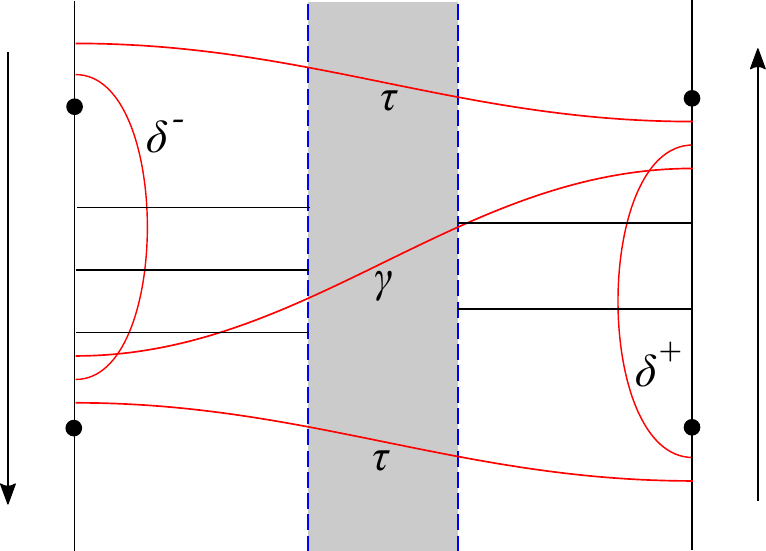}
\end{center}
\caption{Objects of $\scrW(\C^*,\sigma)$, illustrated in the universal cover. Stops are illustrated with solid dots, and the direction of the Reeb flow on the boundary components is illustrated with an arrow. The pullback of the projection of $V$ is contained in the shaded region together with the `legs' corresponding to cylindrical ends. The pullbacks of the objects $\gamma, \tau, \delta^\pm$ of $\scrW(\C^*,\sigma)$ are illustrated in red.}
\label{Fig:PW_objects}
\end{figure}

\begin{prop}\label{prop:cob_rel_groth}
There is a well-defined map
\begin{align*}
\Cob^\unob(X/\C^*) &\to K\left(\scrF(X)^{prop}\right), \quad\text{ sending} \\
[L] & \mapsto \left[\scrY(L)\right]
\end{align*} 
where $\scrY$ denotes the Yoneda embedding. 
\end{prop}
\begin{proof}
Consider the objects $\gamma,\tau,\delta^\pm$ of $\scrW(\C^*,\sigma)$ whose pullbacks to the universal cover are illustrated in Figure \ref{Fig:PW_objects}. 
Our previous argument shows that there are quasi-isomorphisms
\begin{equation}
\label{eqn:triangles}
 \scrY(\delta^\pm,V) \simeq \scrY\left(L_1^\pm\right) \to \left(\scrY\left(L_2^\pm\right) \to \left(\scrY\left(L_3^\pm\right) \to \ldots \right) \right).
 \end{equation}
We observe that there are quasi-isomorphisms 
\[ \left(\delta^- \to \gamma \right) \simeq \tau \simeq \left(\delta^+ \to \gamma \right)\]
in $\scrW(\C^*,\sigma)$. 
It follows that there is a quasi-isomorphism
\[ \left(\scrY(\delta^-,V) \to \scrY(\gamma,V)\right) \simeq \scrY(\tau,V) \simeq \left(\scrY(\delta^+) \to \scrY(\gamma,V) \right)\]
in $\scrF(X)^{prop}$. 
Hence we have
\[ \left[ \scrY(\delta^-,V) \right] = \left[ \scrY(\tau,V) \right] - \left[ \scrY(\gamma,V) \right] = \left[\scrY(\delta^+,V)\right]\]
in the Grothendieck group. 
Combining with \eqref{eqn:triangles} shows that
\[ \sum_i \scrY(L_i^-) = \sum_i \scrY(L_i^+) \]
in $K(\scrF(X)^{prop})$, as required.
\end{proof}

\begin{cor}
If $\scrF(X)^{perf}$ is homologically smooth (e.g., if $X$ admits a homological mirror), then we have a surjective homomorphism
\begin{align*}
\Cob^\unob(X/\C^*) & \to K\left(\scrF(X)^{tw}\right), \quad \text{sending} \\
[L] & \mapsto [L].
\end{align*}
\end{cor}
\begin{proof}
This is immediate from the combination of Proposition \ref{prop:cob_rel_groth} with the fact that the map
\begin{align*}
K\left(\scrF(X)^{tw}\right) &\to K\left(\scrF(X)^{prop}\right) =  K\left(\scrF(X)^{perf}\right) \\
[L] & \mapsto \left[\scrY(L)\right]
\end{align*}
is injective when $\scrF(X)^{perf}$ is smooth.
\end{proof}

\begin{rmk}
The difference between planar and cylindrical cobordism is particularly stark for cobordisms with one positive and one negative end, $L^+$ and $L^-$. 
Biran and Cornea's results show that such a planar cobordism induces a quasi-isomorphism $L^+ \simeq L^-$ in $\scrF(X)$; whereas the best we can hope for from a cylindrical cobordism is an equality $[L^+] = [L^-]$ in $K(\scrF(X)^{tw})$.
\end{rmk}

\subsection{Cobordism versus rational equivalence}\label{Sec:cobrat}

Let $Y$ be an algebraic variety. A $0$-cycle $\Upsilon$ on $Y$ is a finite formal linear combination of closed points $\sum n_i [p_i]$ with $p_i \in Y$ and $n_i \in \Z$.  Let $Z_0(Y)$ denote the free abelian group on the closed reduced points of $Y$, whose elements are then by definition the $0$-cycles. 

A zero-dimensional subscheme $Z\subset Y$ has an associated $0$-cycle $[Z] = \sum m_i [Z_i]$ where the $\{Z_i\}$ are the reduced irreducible components of $Z$ and the $m_i$ their multiplicities.

\begin{defn}\label{defn:rational_equivalence}
A $0$-cycle $\Upsilon$ on $Y$ is rationally equivalent to zero if there are one-dimensional irreducible subvarieties $\scrE_1,\ldots,\scrE_t$ in $\bP^1\times Y$, dominant over $\bP^1$, for which
\begin{equation} \label{eqn:rational_equiv}
\Upsilon = \sum_{j=1}^t \, [\scrE_i(0)] - [\scrE_i(\infty)]
\end{equation}
where $\scrE_i(p)$ is the scheme-theoretic fibre over $p\in \bP^1$ of the induced map $\scrE_i\to\bP^1$.  Two cycles $\Upsilon$ and $\Upsilon'$ are rationally equivalent if $\Upsilon-\Upsilon'$ is rationally equivalent to zero. The quotient of $Z_0(Y)$ by this equivalence relation is the Chow group of $0$-cycles $\CH_0(Y)$.\footnote{The definition generalises straightforwardly to higher dimensional cycles (finite formal linear combinations of $k$-dimensional irreducible subvarieties of $Y$, modulo relations coming from $(k+1)$-dimensional subvarieties of $\bP^1 \times Y$ dominant over $\bP^1$), leading to groups $\CH_k(Y)$.} 
\end{defn}

Rational equivalence is not usually discussed for analytic varieties that are not algebraic; for instance, there are complex non-projective surfaces which contain no curves but have line bundles, so the relation between the Picard group and divisors breaks down.  Nonetheless, the definition makes  sense on a rigid analytic space $Y$: an analytic subspace has well-defined irreducible components with integer multiplicities \cite{Conrad}, and one can consider formal finite linear combinations of points modulo the equivalence relation generated by \eqref{eqn:rational_equiv}. If $Y$ is proper and is the analytification of an algebraic variety $Y^{alg}$, then the rigid analytic GAGA principle shows that any analytic curve in $(\bP^1)^{an}\times Y$ lifts to an algebraic curve in $\bP^1\times Y^{alg}$, recovering the previous definition.

\begin{prop} \label{cor:cobordism_gives_rational_equiv}
Let $V\subset \C^*\times X(B)$ be an unobstructed cobordism with multiplicity one ends $\bL^- = \left\{F_{y^-_1}, \ldots, F_{y^-_n}\right\}$ and $\bL^+ =  \left\{F_{y^+_1}, \ldots, F_{y^+_n}\right\}$. 
Then the effective, reduced $0$-cycles $\sum_{j=1}^n [y^-_j]$ and $\sum_{j=1}^n [y^+_j]$ on $Y(B)$ are rationally equivalent. 
\end{prop}

\begin{proof} By Proposition \ref{prop:support_compactifies}, we have an analytic curve given by the union of those one-dimensional irreducible components of $\scrE^0_V \subset (\bP^1)^{an} \times Y(B)$ which contain the ends (i.e. whose tropicalizations map to a neighbourhood of infinity in $\R\times B$).  Since $V$ has multiplicity one ends, this curve intersects the fibre over $0 \in (\bP^1)^{an}$ in the points $y_j^-$ and the fibre over $\infty \in (\bP^1)^{an}$ in the points $y_j^+$ by the proof of Proposition \ref{prop:support_not_empty}. The result follows.
\end{proof}

\begin{rmk}
If $V$ has ends of higher multiplicity, it is still true that the intersection of $\scrE^0_V$ with the fibre over $0/\infty$ is the union of points $y_j^-/y_j^+$ with multiplicities (cf. Remark \ref{rmk:fibre_mult}). 
However these may be isolated points of $\scrE^0_V$, in which case we can draw no conclusion about rational equivalence. 
We are only able to show that they are transverse intersections of a one-dimensional component of $\scrE^0_V$ with the fibre over $0/\infty$ if the ends have multiplicity one.
\end{rmk}

\subsection{Cobordism versus K-theory\label{Sec:rational_equivalence}}

We say that a symplectic manifold $(X,\omega)$ is \emph{homologically mirror}  to an algebraic variety $Y$ over the Novikov field $\Lambda$ if there is an $A_{\infty}$-quasi-equivalence
\[
\scrF(X,\omega)^{perf} \simeq  \EuD(Y)
\] 
(recall the RHS denotes a $dg$-enhancement of the derived category). 
Such an equivalence manifestly yields an isomorphism of Grothendieck groups of the corresponding categories. 
The Grothendieck group of the derived category of $Y$ is isomorphic to $K(Y)$, the algebraic K-theory of $Y$, so combining with Proposition \ref{prop:cob_rel_groth} we obtain 
\begin{equation}
\label{eqn:cob_versus_k}
 \Cob^\unob(X/\C^*) \to K(\scrF(X)^{perf}) \simeq K(Y).
\end{equation}
It is tempting to wonder when this map might be an isomorphism.

One expects that $\Cob^\unob_\fib(X(B))$ should correspond to the subgroup of $K(Y(B))$ generated by skyscraper sheaves of points. 
The latter is isomorphic to $\CH_0(Y(B))$ via the map $\CH_0(Y(B)) \to K\left(Y(B)\right) $ sending $[y]  \mapsto [\mathcal{O}_y]$. 
This is consistent with Proposition \ref{cor:cobordism_gives_rational_equiv}, which suggests a more direct method of comparing $\Cob^\unob_\fib$ with $\CH_0$.

\begin{rmk}
The map $\CH_0(Y) \to K(Y)$ does not extend to a map $\CH_*(Y) \to K(Y)$. 
Rather, one defines the ``topological filtration'' on the Grothendieck group $K(Y)$, by setting $F_j K(Y)$ to be generated by coherent sheaves whose support has dimension $\leq j$; then it is known that there is a natural graded homomorphism from $\CH_*(Y)$ to the associated graded group of the topological filtration, extending $\CH_0(Y) \to K(Y)$, cf. \cite[Example 15.1.5]{Fulton:intersection_theory}. 
\end{rmk}



\begin{ex} \label{ex:caution}
Consider the tropical affine two-torus constructed in Lemma \ref{lem:mscurvenopol}, which contains a mixed sign tropical curve but no tropical curve. 
Since $Y(B)$ does not contain any analytic divisor (curve), if two elements of $K(Y(B))$ define the same class in $\gr_2 K(Y(B))$ then they differ by an element supported in dimension $0$, and hence have the same first Chern class. 
Since the image of the map $\CH_2(Y(B)) \to \gr_2 K(Y(B))$ is spanned by $[\mathcal{O}_{Y(B)}]$, this means all elements of the image have vanishing first Chern class. On the other hand the mixed sign curve defines a non-trivial class  in $H^1(B;T^*_{\Z}B)$ which lifts to a class $0\neq \eta \in H^1(B;\Aff)$. There is an injection $H^1(B;\Aff) \to H^1(Y(B);\mathcal{O}^*)$ to the Picard group of the two-dimensional analytic space $Y(B)$, so $Y(B)$ has an analytic line bundle $\scrE_\eta$ with non-trivial first Chern class, which cannot lie in the image of the map $\CH_2(Y(B)) \to \gr_2 K(Y(B))$. 
The line bundle $\scrE_\eta$ is mirror to a Lagrangian section $L_\eta$ of $X(B)$, so the classes $[\scrE_\eta] \in K(Y(B))$ and $[L_\eta] \in \Cob^\unob(X(B))$ should correspond under the conjectural isomorphism \eqref{eqn:cob_versus_k}. 
This gives further evidence that the correct mirror to $\Cob^\unob(X(B))$ is $K(Y(B))$ rather than $\CH_*(Y(B))$, since the latter does not `see' $[L_\eta]$.
\end{ex}

\subsection{The mirror to cylindrical cobordism}

The stopped Liouville manifold $(\C^*,\sigma)$ is mirror to the weighted projective line $\bP^1$, so we expect a mirror equivalence
\[ \scrW(\C^* \times X,\sigma \times X) \simeq \EuD\left(\bP^1 \times Y\right)\]
(we remark that in any case where one can prove homological mirror symmetry for $X$ and $Y$ via family Floer theory  \cite{Abouzaid:ICM, Abouzaid:FFFF, Abouzaid:HMS_without_corrections}, one can hope to prove this extension using the modified Fukaya trick from the proof of Proposition \ref{prop:support_compactifies}). 
The objects $\delta^-,\delta^+$ of $\scrW(\C^*,\sigma)$ are mirror to the skyscraper sheaves $\mathcal{O}_0, \mathcal{O}_\infty$ of $\EuD(\bP^1)$ respectively. 
Therefore, the functor 
\[ \scrY(\delta^-,-): \scrW(\C^* \times X, \sigma \times X) \to \scrF(X)^{prop}\]
is mirror to the restriction functor $i_0^*: \EuD(\bP^1 \times Y) \to \EuD(Y)$, where $i_0: Y \hookrightarrow \bP^1 \times Y$ is the inclusion of the fibre over $0 \in \bP^1$. 
Similarly, $\scrY(\delta^+,-)$ is mirror to $i_\infty^*$. 

This suggests the following construction as the mirror to $\Cob^\unob(X/\C^*)$.  
Let $\sim$ denote the equivalence relation on objects of $\EuD(Y)$ generated by (isomorphism of objects together with) $i_0^* \scrE \sim i_\infty^*\scrE$  for any object $\scrE$ of $\EuD(\bP^1 \times Y)$.
The quotient of the set of objects of $\EuD(Y)$ by the equivalence relation $\sim$ is an abelian group $K'(Y)$, with the group operation given by direct sum. 

\begin{lem} \label{lem:recover_K0}
There is a natural isomorphism $K'(Y) \xrightarrow{\sim} K(Y)$, sending $
[\scrE]  \mapsto [\scrE]$.
\end{lem}

The map $K(\mathbb{P}^1) \otimes_\Z K(Y)    \to K(\bP^1 \times Y) $ sending $
\scrF \otimes \scrG  \mapsto \mathrm{pr}_{\bP^1}^* \scrF \otimes \mathrm{pr}_Y^* \scrG$ is an isomorphism \cite[Theorem 4.5]{Manin1969}, and $i_z^*\left(\mathrm{pr}_{\bP^1}^* \scrF \otimes \mathrm{pr}_Y^* \scrG\right) \simeq \scrF_z \otimes \scrG$ for any inclusion $i_z:Y \hookrightarrow \bP^1 \times Y$ of the fibre over $z$; taking $\scrF$ to be a complex of locally free sheaves and $z = 0,\infty$ yields well-definition of the map $K'(Y) \to K(Y)$. We omit the presumably well-known proof that this is an isomorphism. Lemma \ref{lem:recover_K0} gives further evidence that the natural mirror to $\Cob^{\unob}(X/\C^*)$ is $K(Y)$, and that the map $\Cob^{\unob}(X/\C^*) \to K(\scrF(X)^{perf})$ might be an isomorphism when the mirror to $X$ is algebraic.

\bibliographystyle{amsalpha}
\bibliography{mybibold}

\end{document}